\documentclass{svjour3}
\usepackage{fix-cm}
\usepackage[a4paper]{geometry}

\usepackage{mathtools}
\usepackage{xcolor}
\usepackage{graphicx}
\usepackage[numbers]{natbib}
\usepackage{subfig}
\usepackage{caption}
\usepackage{amssymb}
\usepackage{amsmath}

\usepackage{amsxtra} 
\usepackage{amstext}
\usepackage{amssymb}
\usepackage{amsfonts}
\usepackage{multicol}
\usepackage{setspace}

\usepackage[version=4]{mhchem}
\usepackage{float}
\usepackage[super]{cite}
\usepackage{url}

\usepackage{hyperref}

\newcommand{\set}[1]{\left\{#1\right\}}

\newcommand{\R}{\mathbb R}

\newcommand{\C}{\mathcal C}

\newcommand{\Z}{\mathbb Z}
\newcommand{\Q}{\mathbb Q}

\def\R{\mathbb{R}}

\def\C{\mathbb{C}}
\def\Z{\mathbb{Z}}
\def\Q{\mathbb{Q}}

\def\z{\mathbb{Z}_p}
\def\q{\mathbb{Q}_p}

\def\-{\setminus}

\def \ord{\mathrm{ord}}
\def \tes{\mathcal{D}}

\newcommand\p[1]{\left(  #1 \right) }

\newcommand\Nor[1]{ \left\Vert #1 \right\Vert}

\newcommand\no[1]{ \left\vert #1 \right\vert}

\newcommand\scalemath[2]{\scalebox{#1}{\mbox{\ensuremath{\displaystyle #2}}}}

\DeclareUnicodeCharacter{2212}{−}

\usepackage[compatibility=false]{caption}

\usepackage{caption}
\begin{document}
\title{A \texorpdfstring{$p$}{p}-adic Reaction-Diffusion Model of Branching Coral Growth and Calcification Dynamics}

\author{Angela Fuquen-Tibatá \and Yuriria Cortés-Poza \and J. Rogelio Pérez-Buendía}

\institute{Angela Fuquen-Tibatá\at IIMAS, Unidad Académica de Yucatán, Universidad Nacional Autónoma de México (UNAM), Yuc., México\\
\email{angela.fuquen@iimas.unam.mx}
\and
Yuriria Cortés-Poza\at IIMAS, Unidad Académica de Yucatán, Universidad Nacional Autónoma de México (UNAM), Yuc., México \\
\email{yuriria.cortes@iimas.unam.mx}
\and 
J. Rogelio Pérez-Buendía \at IXM-Secihti Centro de Investigación en Matemáticas (Cimat), Unidad Mérida, Yuc., México\\
\email{rogelio.perez@cimat.mx}
}
\date{\today }

\maketitle

\begin{abstract}
Coral colonies exhibit complex, self-similar branching architectures shaped by biochemical interactions and environmental constraints. 
To model their growth and calcification dynamics, we propose a novel reaction–diffusion framework defined over $p$-adic ultrametric spaces. 
The model incorporates biologically grounded reactions involving calcium and bicarbonate ions, whose interplay drives the precipitation of calcium carbonate (\ce{CaCO3}). 
Nonlocal diffusion is governed by the Vladimirov operator over the $p$-adic integers, naturally capturing the hierarchical geometry of branching coral structures. 
Discretization over $p$-adic balls yields a high-dimensional nonlinear ODE system, which we solve numerically to examine how environmental and kinetic parameters—particularly \ce{CO2} concentration—influence morphogenetic outcomes. 
The resulting simulations reproduce structurally diverse and biologically plausible branching patterns. 
This approach bridges non-Archimedean analysis with morphogenesis modeling and provides a mathematically rigorous framework for investigating hierarchical structure formation in developmental biology.
\end{abstract}

\keywords{Coral growth modeling \and p-adic analysis \and morphogenesis model \and reaction–diffusion systems \and hierarchical morphogenesis}
\newpage

\section{Introduction}
\label{sec:intro}

Coral reefs are among the most biologically diverse and economically significant ecosystems, providing food, habitat, and coastal protection to numerous marine species \citep{fisher_species_2015,moberg_ecological_1999, YUAN2024, carlot_coral_2023}. However, these ecosystems are increasingly threatened by climate change, ocean acidification, and pollution see e.g., \citep{hoegh-guldberg_coral_2017,hughes_coral_2017, Terry2018}. A deeper understanding of coral growth dynamics—particularly under varying environmental conditions—is essential for developing effective strategies for their conservation and restoration.

Mathematical models of coral growth have often been formulated within Euclidean frameworks, primarily to investigate the influence of environmental factors. 
For instance, studies have focused on the effect of external water flow on coral morphology. 
Mistr and Bercovici \citep{mistr_theoretical_2003} developed a porous-flow model showing that unidirectional currents can orient and space coral colonies, while Merks et al. in \citep{MERKS2003} used computational simulations to demonstrate that hydrodynamic transport of nutrients can drive morphological plasticity. 
These works successfully highlight the role of external flow on large-scale reef architecture. 
Nakamura and Nakamori \citep{nakamura_geochemical_2007} proposed a geochemical diffusion model in which reef topography emerges from  the coupling between carbonate chemistry, light-enhanced calcification, and the diffusive transport of dissolved inorganic carbon. 
More recently, Llabrés et al. \citep{Llabrs2024} developed an agent-based numerical framework for clonal coral growth,  illustrating how simple local growth rules at the polyp scale can reproduce the diversity of colony morphologies observed in scleractinian species. However, they do not examine the internal, biochemical regulation of calcification that operates within the coral's own tissue and skeletal structure.

Biologically, calcification in corals occurs in the extracellular calcifying medium (ECM) beneath the calicoblastic epithelium \citep{Tambutt2011, venn2025,crovetto_spatial_2024}. In our model, this process is represented by a reaction term within a reaction–diffusion framework. We consider branching corals in which ions diffuse through the gastrovascular canals that interconnect polyps along each branch. The reaction term represents the precipitation of calcium carbonate, representing the effective transformation of these ions into skeletal material. The assumption of internal diffusion through the gastrovascular system is supported by experimental evidence showing circulation of fluids between polyps in Acropora cervicornis \citep{Gladfelter1983}.

This work introduces a $p$-adic reaction–diffusion model designed to capture aspects of branching morphogenesis in corals within a non-Archimedean, ultrametric space. The model is formulated over the ring of $p$-adic integers $\mathbb{Z}_p$, whose hierarchical, tree-like topology provides an abstract but natural mathematical substrate for representing ramified biological structures. We incorporate a simplified set of reactions representing essential steps in calcium carbonate (\ce{CaCO3}) formation, and we employ the Vladimirov operator ${\boldsymbol{D}}^{\alpha}$—a $p$-adic pseudodifferential analog of the fractional Laplacian—to model nonlocal diffusion across the branching structure. The concentrations of carbonate (\ce{CO3^{2-}}), calcium (\ce{Ca^{2+}}), and calcium carbonate (\ce{CaCO3}) are described by real-valued functions defined on $\mathbb{R} \times \mathbb{Z}_p$, with time as a real variable and space represented by a $p$-adic integer.
Diffusion in ultrametric spaces has been widely explored in the context of hierarchical energy landscapes \citep{Avetisov_2002, KHRENNIKOV2021} and, more recently, in reaction–diffusion systems on networks and hierarchical models \citep{zuniga-galindo_reaction-diffusion_2020, chacon-cortes_turing_2023, zuniga-galindo_hierarchical_2024}. Building on these developments, our study applies ultrametric diffusion to a branching biological structure—corals. To our knowledge, this represents among the first attempts to use a $p$-adic framework to investigate a living branching system, 
offering a novel mathematical perspective on the internal chemical dynamics underlying coral calcification.

We discretize the model by partitioning $\mathbb{Z}_p$ into a finite collection of disjoint balls, each representing a coral branch. The resulting discretized problem reduces to solving a system of coupled ordinary differential equations that describe the dynamics of calcification within each branch. In this formulation, the Vladimirov operator encodes the diffusive coupling between branches, capturing the exchange of chemical species across the coral network. Numerical simulations using standard ODE solvers allow us to follow the temporal evolution of concentrations within each branch and to explore how parameter variations affect the calcification dynamics.

The model incorporates a dynamic bifurcation rule triggered by local \ce{CaCO3} accumulation, as well as a biologically motivated halting condition based on saturation thresholds. These features enable us to simulate the progressive development and eventual cessation of coral growth in response to environmental constraints.

It is important to emphasize that the prime parameter $p$ in our formulation is a mathematical requirement, not a biological constraint on coral branching. 
The prime $p$ defines the ultrametric tree structure of $\mathbb{Z}_p$, providing a rigorous scaffold on which nonlocal diffusion can be consistently formulated. 
The framework applies equally to any prime $p$, offering a general representation of hierarchical organization. 
Our aim is to establish a minimal yet mathematically consistent framework that introduces $p$-adic ultrametric diffusion into morphogenesis modeling.

Compared to Euclidean approaches, our framework offers three key advantages. First, incorporating explicit reaction kinetics provides a biologically grounded mechanism for calcification (see~\autoref{Biochemical_model}). Second, the ultrametric structure of $\mathbb{Z}_p$ naturally encodes branching geometries and enables diffusion across all hierarchical levels (see ~\autoref{Numerical_sol}). Third, the use of $p$-adic pseudodifferential operators expands the toolkit available for modeling self-organizing processes in hierarchical biological systems.

This article is organized as follows. 
Section~\ref{Mat_Biol_back} introduces the fundamental concepts of $p$-adic variables and their relevance for modeling ramified structures and reviews essential aspects of coral biology. Section~\ref{sec:model} presents the construction of our reaction-diffusion model. Section~\ref{Numerical_sol} details the discretization procedure, numerical solutions, and coral growth simulations. Technical aspects of $p$-adic analysis and operator theory are deferred to ~\autoref{appe}.

%------------Biological background--------------------%
\section{Mathematical and Biological Background}\label{Mat_Biol_back}

\subsection{\texorpdfstring{$p$}{p}-adic Variables in the Study of Ramifications}\label{Intro_padics}

Branching structures are pervasive in trees, vascular systems, river networks, and coral colonies. Modeling such systems requires geometric information and a way to encode hierarchical relationships among branches. The space of $p$-adic numbers provides a natural mathematical framework, offering a topological and algebraic structure that reflects such hierarchies. Recent works in mathematical biology have begun to explore the potential of $p$-adic analysis for modeling complex biological systems with ramified architectures; see \citep{dragovich_p-adic_2021} for an overview.

The field of $p$-adic numbers, denoted $\mathbb{Q}_p$, is defined as the completion of the rational numbers $\mathbb{Q}$ with respect to the $p$-adic absolute value.  For $x \neq 0$, this valuation is given by $|x|_p = p^{-m}$, where $m$ is the unique integer such that $x = p^{m} \frac{a}{b}$ with $a,b$ integers not divisible by $p$. This induces a non-Archimedean (ultrametric) distance:
\begin{equation}
    d_p(x, y) = |x - y|_p,
\end{equation}
which groups points not by proximity in the Euclidean sense but by the similarity of their divisibility structure.

The closed unit ball centered at $0$ under this metric is the set of \emph{$p$-adic integers}, denoted $\mathbb{Z}_p$. 
Each $x \in \mathbb{Z}_p$ admits a unique expansion:
\begin{equation}\label{expan}
    x = a_0 + a_1 p + a_2 p^2 + \cdots, \quad a_i \in \{0, 1, \ldots, p-1\},
\end{equation}
which naturally defines a tree-like structure: each digit $a_i$ determines the branch taken at level $i$ in an infinite $p$-ary tree. 
The first digit $a_0$ corresponds to one of $p$ branches from the root, each of which further splits according to $a_1$, and so on. 
Every infinite path in this tree uniquely identifies a $p$-adic integer.

\begin{figure}[ht]
    \centering
    \includegraphics[width=0.7\linewidth]{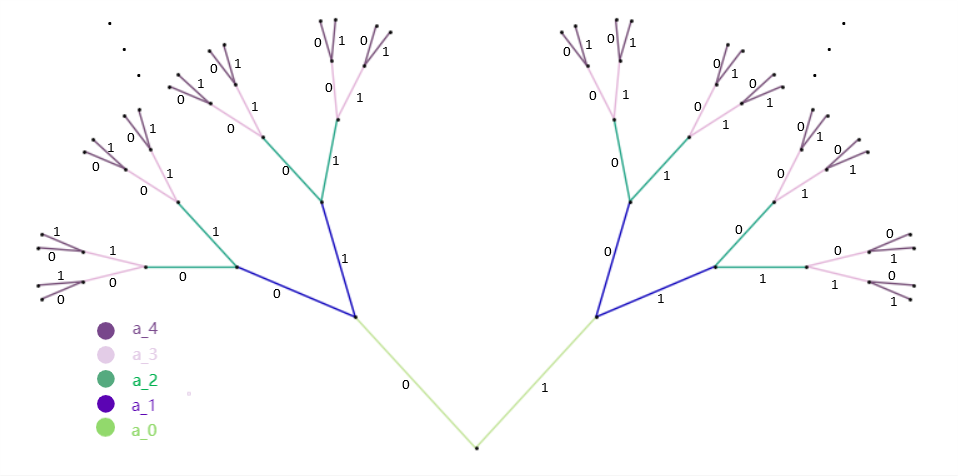}
    \caption{First five levels of the binary tree representing $\mathbb{Z}_2$ via the expansion~\eqref{expan}. Each node corresponds to a coefficient $a_i \in \{0,1\}$.}
    \label{arbol}
\end{figure}

For $p=2$, the tree begins with two branches ($a_0 = 0$ or $1$), and each subsequent level again bifurcates. Figure~\ref{arbol} illustrates this process for five levels. Analogous trees for $p=3$ and $p=5$ are shown in~\autoref{branching_structures}.

\begin{figure}[ht]
\centering
\subfloat[First four levels of the $3$-adic tree.\label{Z_3}]{
  \includegraphics[width=0.45\linewidth]{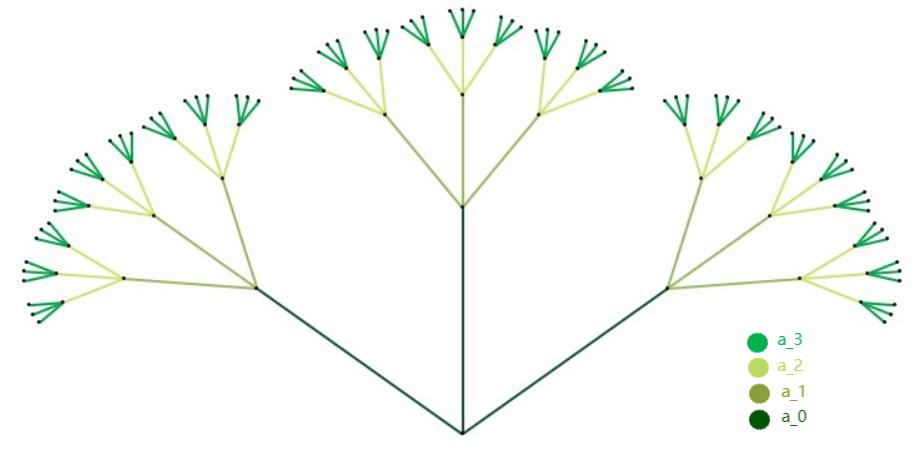}
}
\hfill
\subfloat[First three levels of the $5$-adic tree.\label{Z_5}]{
  \includegraphics[width=0.45\linewidth]{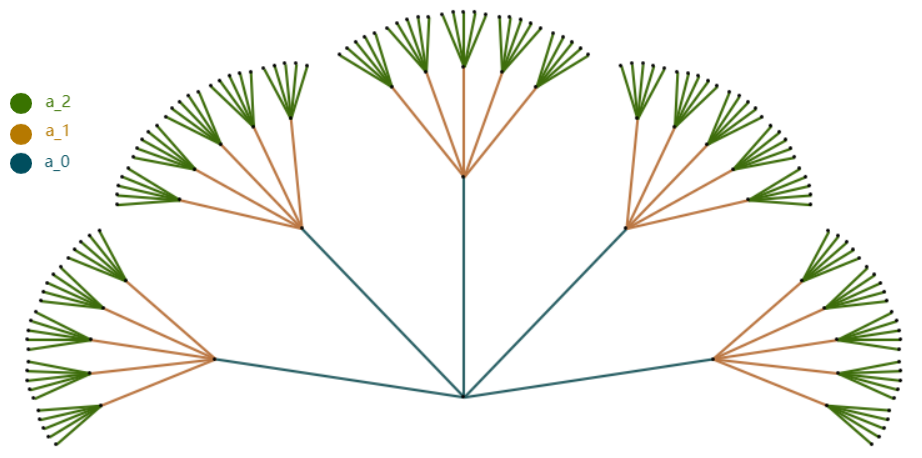}
}
\caption{Hierarchical trees corresponding to $\mathbb{Z}_3$ and $\mathbb{Z}_5$. Each level encodes a digit $a_i \in \{0,\dots,p-1\}$.}
\label{branching_structures}
\end{figure}

Each level of the $p$-adic tree corresponds to a partition of $\mathbb{Z}_p$ into disjoint balls of smaller radius. 

We take the set $G_m$ of $p$-adic integers defined in~\autoref{Intro_padics}, which indexes the $p^m$ disjoint balls $B_{-m}(i)$ representing the branches of the coral.

Then $\mathbb{Z}_p$ can be written as the disjoint union:
\begin{equation}
\mathbb{Z}_p = \bigsqcup_{a \in G_m} B_{-m}(a),
\end{equation}
where $B_{-m}(a)$ denotes the closed ball of radius $p^{-m}$ centered at $a$. These nested partitions reflect the hierarchical structure of $\mathbb{Z}_p$ and can be visualized as successive subdivisions of the tree at each level, see~\autoref{nested_balls}.

\begin{figure}[ht]
\centering
\subfloat[$\mathbb{Z}_2$ as a union of dyadic balls.\label{nested_Z2}]{
  \includegraphics[width=0.45\linewidth]{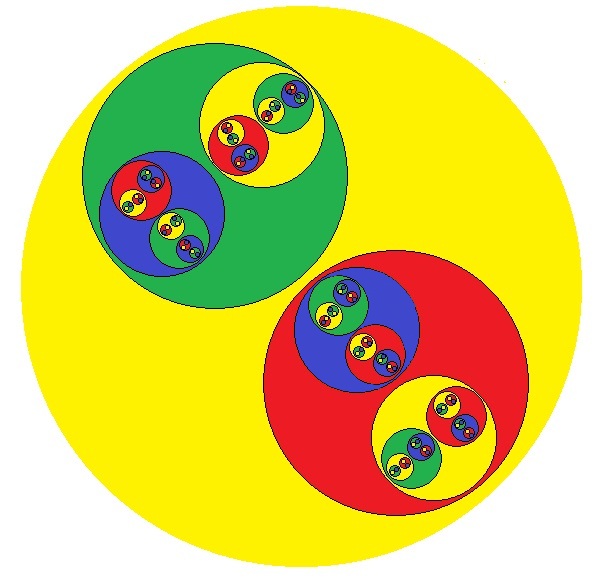}
}
\hfill
\subfloat[$\mathbb{Z}_5$ as a union of five-ary balls.\label{nested_Z5}]{
  \includegraphics[width=0.45\linewidth]{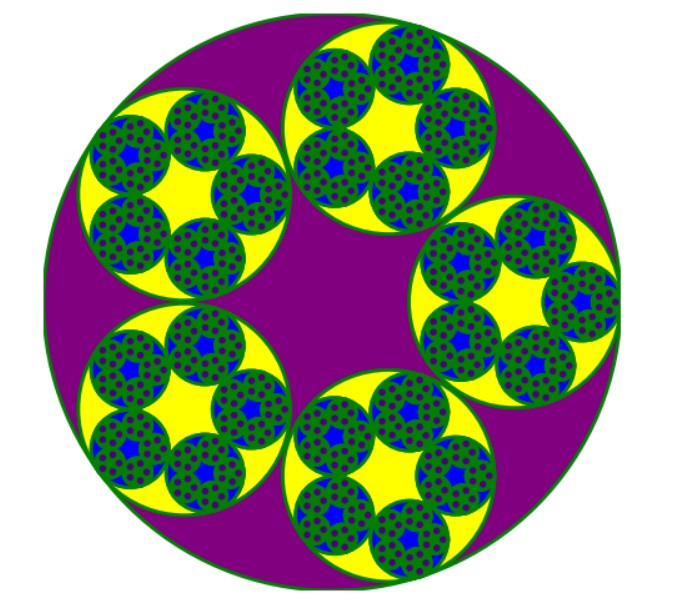}
}
\caption{Fractal-like nesting of $p$-adic balls: each ball contains $p$ disjoint sub-balls at the next level.}
\label{nested_balls}
\end{figure}

In our model, we use $\mathbb{Z}_p$ to represent the branching directions of coral structures. The ultrametric distance $d_p$ quantifies the similarity between branches in terms of shared bifurcation levels. Two branches that share more common digits (i.e., lie in the same larger ball) are considered "closer" than those that diverge earlier in the expansion.

Diffusion in this hierarchical space is described using the Vladimirov operator, a pseudodifferential operator defined for $\alpha > 0$ by:
\begin{equation}
\boldsymbol{D}^\alpha u(x) := \frac{1 - p^{\alpha}}{1 - p^{-\alpha - 1}} \int_{\mathbb{Q}_p} \frac{u(y) - u(x)}{|x - y|_p^{1 + \alpha}},dy.
\end{equation}

The restriction of $\boldsymbol{D}^\alpha$ to the space of locally constant functions supported in $\mathbb{Z}_p$ is denoted by $\overline{\boldsymbol{D}}^\alpha$ (see~\autoref{appe}).

In this context, $\boldsymbol{D}^\alpha$ acts as a $p$-adic analog of the fractional Laplacian, capturing the propagation of information (in our case, chemical species) across all tree scales.

The $p$-adic formalism offers multiple advantages for modeling branching:
\begin{itemize}
    \item \textit{Hierarchical geometry}: the metric structure of $\mathbb{Z}_p$ naturally encodes tree-like relationships.
    \item \textit{Nonlocal diffusion}: The Vladimirov operator enables chemical interactions across the entire hierarchy, not just among neighboring branches.
    \item \textit{Compactness and scalability}: the compactness of $\mathbb{Z}_p$ allows for efficient discretization and simulation.
\end{itemize}

These properties make $p$-adic spaces a powerful setting for modeling coral branching and other forms of biological morphogenesis with recursive spatial organization.

\subsection{Coral Biology}\label{coral_review}

Corals are marine invertebrates belonging to the phylum Cnidaria, along with jellyfish and sea anemones. Among their most remarkable features is their ability to build extensive calcium carbonate (\ce{CaCO3}) skeletons, which form the structural basis of coral reefs. These reef systems are among the planet's most diverse and productive ecosystems, supporting approximately 25\% of marine species~\citep{fisher_species_2015}.

Coral colonies comprise numerous polyps—small, genetically identical animals—connected by a tissue layer called the coenosarc, through which nutrients are distributed. Calcification occurs at the base of each polyp in the extracytoplasmic calcifying fluid (ECF), a microenvironment whose ionic composition is regulated by the coral~\citep{tambutte_coral_2011}. Over time, the accumulation of skeletal material enables the colony to grow in size and complexity, forming massive reef structures.

Coral growth forms vary across species, including branching, massive, encrusting, and foliaceous morphologies. This study focuses on branching corals exhibiting tree-like architectures with rapid expansion and high surface-area-to-volume ratios. Representative species include \textit{Acropora palmata}, \textit{Acropora cervicornis}, and \textit{Pocillopora damicornis}.

Beyond their structural role, coral reefs are critical to nutrient cycling, wave attenuation, and shoreline protection. They underpin local economies via fisheries, tourism, and marine product harvesting~\citep{moberg_ecological_1999}. However, these ecosystems are susceptible to environmental stressors, particularly ocean warming, acidification, and pollution. Such factors have led to widespread coral bleaching and mortality, threatening the ecological and economic functions of reefs~\citep{hoegh-guldberg_coral_2017}.

Understanding the mechanisms that govern coral growth is essential for predicting their response to changing environmental conditions. Mathematical modeling of these processes—especially calcification and structural development—offers a powerful tool for exploring coral resilience and guiding conservation strategies. In this work, we integrate biochemical knowledge of coral calcification with a hierarchical spatial framework to capture coral morphogenesis's chemical and geometric features.

\section{The Proposed Model}\label{sec:model}

The previous section reviewed the mathematical and biological background underlying our approach. 
We now develop the proposed biochemical $p$-adic reaction–diffusion model, which constitutes the core contribution of this work. 
This model integrates key aspects of coral calcification chemistry with the ultrametric diffusion framework introduced above. 
It provides a minimal yet mathematically consistent representation of ion transport and precipitation processes occurring within the coral's internal branching network.

\subsection{Biochemical \texorpdfstring{$p$}{p}-adic Reaction–Diffusion Model for Coral Calcification}\label{Biochemical_model}
We introduce a $p$-adic reaction–diffusion model that captures the dynamics of coral calcification through coupled biochemical reactions and hierarchical diffusion.

Calcium carbonate (\ce{CaCO3}) can be formed via the reaction of calcium ions with either carbonate (\ce{CO3^{2-}}) or bicarbonate (\ce{HCO3^{-}}) ions. Since bicarbonate is the dominant form of dissolved inorganic carbon in seawater~\citep{cole_dissolved_2014}, we base our model on the latter pathway. The following chemical reactions are considered:
\begin{align}\label{Chemical1}
&\ce{CO2 + H2O + CO3^{2-}} \xrightarrow{k_1} 2\ce{HCO3^{-}}, \\
\label{Chemical2}
&\ce{Ca^{2+} + 2HCO3^{-}} \xrightarrow{k_2} \ce{CaCO3 + H2O + CO2},
\end{align}
As described in~\citep{gattuso_photosynthesis_1999}, Reaction~\eqref{Chemical1} reflects the generation of bicarbonate from carbonate and carbon dioxide in seawater, while Reaction~\eqref{Chemical2} accounts for the precipitation of calcium carbonate through the interaction of calcium with bicarbonate. These two reactions represent the key chemical processes underlying coral calcification, linking dissolved inorganic carbon to the formation of solid CaCO$_3$. Under typical seawater conditions, bicarbonate (\ce{HCO3^-}) is the dominant inorganic carbon species, in agreement with marine carbonate chemistry measurements~\citep{MILLERO2008}.

We restrict our model to these reactions, assuming optimal and constant environmental conditions (e.g., temperature, pH). This choice reflects the fact that they capture the core pathway of calcium carbonate formation while keeping the formulation analytically and numerically tractable. By focusing on this essential mechanism, the model provides a transparent setting in which the influence of $p$-adic ultrametric diffusion on morphogenesis can be rigorously assessed. At the same time, because \ce{CO2} modulates bicarbonate availability, the formulation indirectly incorporates the effect of environmental factors—particularly temperature—through the initial concentration of \ce{CO2}~\citep{stips_causal_2016}. More elaborate processes—such as water flow, extended carbonate equilibria, or coral respiration—can be incorporated modularly in future extensions, but our goal here is to establish a minimal and mathematically consistent foundation that highlights the methodological contribution of the $p$-adic framework. In this sense, the model remains closed in terms of dynamic variables but semi-open to the environment via parameter sensitivity.

We denote the concentrations of the involved species as functions of $(t,x)$:
\begin{align*}
z &= [\ce{CO2}], & u &= [\ce{CO3^{2-}}], & s &= [\ce{HCO3^{-}}], & v &= [\ce{Ca^{2+}}], & w &= [\ce{CaCO3}].
\end{align*}

Here, $x \in \mathbb{Z}_p$ represents branch position in the coral structure, and $t \in \mathbb{R}^+$ denotes time. Each function is $C^1$ in time and locally constant in the $p$-adic spatial variable; that is,
\[
\left\{
\varphi : \mathbb{R}^+ \times \mathbb{Z}_p \to \mathbb{R}^+ \,\middle|\,
\begin{array}{l}
\varphi(\cdot, x) \in C^1(\mathbb{R}^+) \text{ for each } x \in \mathbb{Z}_p, \\
\varphi(t, \cdot) \in \tes_{\mathbb{R}}(\mathbb{Z}_p) \text{ for each } t \in \mathbb{R}^+
\end{array}
\right\},
\]
where $\tes_{\mathbb{R}}(\mathbb{Z}_p)$ denotes the space of real-valued, locally constant functions supported in $\mathbb{Z}_p$ (see ~\autoref{appe}). All species are confined to the coenosarc, which we model as a closed reaction--diffusion
domain with no external ion supply or loss. Water acts as the continuous solvent within the
coenosarc, and we assume its concentration $[\ce{H2O}]$ is spatially uniform and constant.
Consequently, we absorb it into the first reaction's rate constant by defining
\[k_1' = k_1 [\ce{H2O}],\]

Following the law of mass action~\citep[Chapter 6]{murray_mathematical_2001}, the reaction kinetics are governed by:
\begin{align}
\frac{\partial z}{\partial t} &= -k_1'  u z + k_2 v s^2, \label{Eq1} \\
\frac{\partial u}{\partial t} &= -k_1'  u z, \label{Eq3} \\
\frac{\partial s}{\partial t} &= k_1'  u z - k_2 v s^2, \label{Eq4} \\
\frac{\partial v}{\partial t} &= -k_2 v s^2, \label{Eq5} \\
\frac{\partial w}{\partial t} &= k_2 v s^2. \label{Eq6}
\end{align}

where $k_1',k_2 >0$ are the reaction rate constants. We assume uniform initial concentrations: $z(0,x)=z_0$, $u(0,x)=u_0$, $v(0,x)=v_0$, and $s(0,x)=w(0,x)=0$. The latter two conditions reflect that no reaction products (\ce{HCO3^{-}} and \ce{CaCO3}) are present at time zero, ensuring consistency with the assumed initial absence of chemical products.

Under these assumptions, several conservation identities allow us to eliminate variables algebraically:
\begin{align*}
z(t,x) &= z_0 - s(t,x), \\
s(t,x) &= v(t,x)-u(t,x)+u_0-v_0,
\end{align*}

Substituting these identities into the system yields a reduced formulation involving only $u$, $v$, and $w$:
\begin{align}
\frac{\partial u}{\partial t} &= -k_1' u (u - v + z_0 - c), \\
\frac{\partial v}{\partial t} &= -k_2 v (v - u + c)^2, \\
\frac{\partial w}{\partial t} &= k_2 v (v - u + c)^2,
\end{align}

The constant $c:=u_0-v_0$ is fixed for all $(t,x)$ and represents the initial imbalance between carbonate and calcium concentrations. Since in normal conditions $\ce{Ca^2+}$ concentration is greater than $\ce{CO3^2-}$ concentration (see \citep{MILLERO2008}), $c$ is taken negative.

This system is strongly nonlinear, involving cubic and quadratic terms in the variables. Analytical solutions are generally not tractable, and numerical methods become necessary for exploring the model's dynamic behavior.

We use the Vladimirov operator $\boldsymbol{D}^\alpha$, a nonlocal pseudodifferential operator defined over $p$-adic domains, to incorporate spatial diffusion. Since we work with locally constant functions supported on $\mathbb{Z}_p$, we consider a modified operator, defined as the standard Vladimirov integral minus a constant term, which arises from the restriction to this function space (see \autoref{appe}). We denote this modified operator by $\overline{\boldsymbol{D}}^\alpha$ for simplicity.

Assuming that only \ce{CO3^{2-}} and \ce{Ca^{2+}} diffuse, we obtain the full reaction-diffusion system:
\begin{align}
\frac{\partial u}{\partial t} &= -d_1 \overline{\boldsymbol{D}}^\alpha u - k_1' u (u - v + z_0 - c), \label{diff_u} \\
\frac{\partial v}{\partial t} &= -d_2 \overline{\boldsymbol{D}}^\alpha v - k_2 v (v - u + c)^2, \label{diff_v} \\
\frac{\partial w}{\partial t} &= k_2 v (v - u + c)^2. \label{diff_w}
\end{align}

Here, $d_1>0$ and $d_2>0$ are diffusion coefficients for carbonate and calcium, respectively. The system~\eqref{diff_u}–\eqref{diff_w} couples chemical kinetics with spatial diffusion on a hierarchical ultrametric structure.

\subsubsection{Nondimensionalization}

We apply a standard nondimensionalization procedure~\citep{segel_simplification_1972} to simplify the system and reduce the number of parameters. For simplicity, we scaled $u,v,$ and $w$ with the same characteristic concentration and let:
\begin{equation}\label{scaling}
    u = \frac{d_1}{k_1'} \, \overline{u}, \quad
v = \frac{d_1}{k_1'} \, \overline{v}, \quad
w = \frac{d_1}{k_1'} \, \overline{w}, \quad
t = \frac{\overline{t}}{d_1}.
\end{equation}

Substituting these expressions into equations~\eqref{diff_u}--\eqref{diff_w} and dropping the overlines for simplicity, we obtain the dimensionless system:
\begin{equation}\label{R-D_system}
\begin{cases}
\frac{\partial u}{\partial t} &= -\overline{\boldsymbol{D}}^\alpha u - u(u - v + \sigma - \beta), \\
\frac{\partial v}{\partial t} &= -d \overline{\boldsymbol{D}}^\alpha v - \eta v(v - u + \beta)^2, \\
\frac{\partial w}{\partial t} &= \eta v(v - u + \beta)^2,
\end{cases}    
\end{equation}

with parameters:
\[
d = \frac{d_2}{d_1}, \quad
\eta = \frac{d_1k_2}{(k_1')^2}\quad
\beta = c \frac{k_1'}{d_1}, \quad
\sigma = z_0 \frac{k_1'}{d_1}.
\]

The nondimensional system~\eqref{R-D_system} captures the essential features of coral calcification using only four parameters, facilitating parametric exploration under different environmental and biochemical scenarios.
Notice that \( d > 0 \), 
since it is the ratio of diffusion coefficients. Similarly, \( \eta \) and \( \sigma \) are also positive as they are defined as the product of positive values and because $c<0$, we have $\beta<0$. 

\begin{remark}\label{remark_nondimensionalization}
    The state variables \(u(t,x)\), \(v(t,x)\) and \(w(t,x)\) are dimensionless, scaled concentrations. Throughout the manuscript, for brevity, we use the term \emph{concentration} to refer to these scaled quantities. Physical concentrations are recovered via \eqref{scaling}
\end{remark}

\subsubsection{Local Stability of the Reaction System}

As $w$ does not affect the dynamics of the other variables, the stability analysis focuses on the first two equations. These equations admit four equilibrium points, obtained as solutions of:
\begin{equation}\label{equations_equilib_points}
\begin{cases}
    f(u,v)=-u(u - v + \sigma - \beta) = 0, \\
    g(u,v)=-\eta v (v - u + \beta)^2 = 0.
\end{cases}
\end{equation}
The equilibrium points are:
\begin{enumerate}
    \item $(0, 0)$,
    \item $(0, -\beta)$.

\end{enumerate}
Only equilibrium points with non-negative concentrations are considered, as negative concentrations are not physically meaningful in chemical systems.
The Jacobian matrix is:
\begin{equation}
 \scalemath{0.8}{
    \begin{bmatrix}
        -(u - v + \sigma - \beta)-u &  u\\
        2\eta v (v - u + \beta) & -\eta  (v - u + \beta)^2 - 2\eta v (v - u + \beta)
    \end{bmatrix}  }
\end{equation}

\begin{enumerate}
    \item The eigenvalues of the Jacobian at the origin \( (0, 0) \) are:
    \[
    \lambda_1 = -(\sigma - \beta), \qquad \lambda_2 = -\eta \beta^2.
    \]
    
     Since \( (\sigma - \beta) > 0 \) and \(\eta \beta^2 > 0\), then both eigenvalues are negative, and the origin is locally asymptotically stable.
       
    \item The equilibrium point \( (0, -\beta) \) has eigenvalues
    \[
    \lambda_1 = - \sigma, \qquad \lambda_2 = 0,
    \]
    with \(  \sigma > 0 \), so the point is non-hyperbolic. Center manifold theory shows that \( (0, -\beta) \) is locally asymptotically stable.

\end{enumerate}

\section{Numerical Solutions}\label{Numerical_sol}

In this section, we present numerical solutions for our system and propose a condition for ramification to simulate coral growth. Varying the parameters, we show how some factors can affect the branching and calcification of the coral. To perform the simulation, we discretize the space of our functions.

\subsection{Discretization of the system}

For a fixed $t$, the functions in our model belong to the space of locally constant functions supported in $\mathbb{Z}_p$, denoted by $\tes_{\mathbb{R}}(\mathbb{Z}_p)$. To visualize the behavior of the solutions, we discretize the system as follows:

Since the number of branches in the coral is finite, we can fix $m \in \mathbb{N}$ and a prime $p$ such that the total number of branches in the coral is $p^m$.

We take the set $G_m$ of $p$-adic integers defined in~\autoref{Intro_padics} in the form:
\begin{align*}
i = i_0 + i_1p + \dots + i_{m-1}p^{m-1},
\end{align*}
where $i_k \in \{0,1,\dots,p-1\}$ for each $k \in \{0,1,\dots,m-1\}$.

We assign to each branch in the coral a ball centered at $i\in G_m$ with radius $p^{-m}$. For $t\ge 0$, consider the concentration functions $u,v$ and $w$ belonging to the space of locally constant functions supported on balls of radius $p^{-m}$ with center in $G_m$; this space is denoted by $X_m$. The set $\{\boldsymbol{1}_{B_{-m}(i)}(x)\}_{i\in G_m}$ forms an orthonormal basis of the space $X_m$ of functions that are locally constant on balls of radius $p^{-m}$ and vanish outside the union of these balls.

Each function can then be written as:
\begin{align*}
   u(t,x) &= \sum\limits_{i\in G_m}u(t,i)\boldsymbol{1}_{B_{-m}(i)}(x), \\
   v(t,x) &= \sum\limits_{i\in G_m}v(t,i)\boldsymbol{1}_{B_{-m}(i)}(x),  \\
   w(t,x) &= \sum\limits_{i\in G_m}w(t,i)\boldsymbol{1}_{B_{-m}(i)}(x)
\end{align*}
where $u(t,i), v(t,i),  w(t,i)$ are continuously differentiable functions for each $i \in G_m$. These functions measure the concentrations of \ce{CO3}, \ce{Ca}, and \ce{CaCO3} on $B_{-m}(i)$, which models one branch of the coral.

The Vladimirov operator restricted to functions in the space $X_m$ is represented by a matrix $L^{\alpha}=[L_{ij}^{\alpha}]$ with size $\#G_m\times \#G_m$. For \(\alpha=2\):

\begin{equation}\label{alpha_2}
    L_{ij}^{2}=
    \begin{cases}
        \frac{p^{3-m}(1-p^2)}{p^3-1}\frac{1}{\no{i-j}_p^{3}} &\text{ if } i\neq j\\
       -\frac{p^{3-m}(1-p^2)}{p^3-1}\displaystyle\sum_{\substack{k\in G_m\\k\ne i}}\frac{1}{|i-k|_p^{3}} &\text{ if } i= j.
    \end{cases}
\end{equation}

The explicit form of the matrix $L^{\alpha}$ is derived in~\autoref{appe}, where its construction and interpretation are discussed in detail.

We denote  
\begin{align}
    f(u,v) &= -{u}({u} - {v} + \sigma - \beta), \\
    g(u,v) &= -\eta {v}({v} - {u} + \beta)^2, \\
    h(u,v) &= \eta{v}({v} - {u} + \beta)^2.
\end{align}
This spatial discretization transforms the original reaction-diffusion PDE system into a coupled system of $3p^m$ ordinary differential equations of the form:

\begin{align}\label{Sdiscret}
\begin{split}
    \frac{\partial u(t,i)}{\partial t} &= -\frac{1-p^{\alpha}}{1-p^{-\alpha-1}}\left(\sum_{\substack{j\in G_m\\j\ne i}}\frac{p^{-m}u(t,j)}{|i-j|_p^{\alpha+1}} - u(t,i)\sum_{\substack{j\in G_m\\j\ne i}}\frac{p^{-m}}{|i-j|_p^{\alpha+1}}\right) \\
&\quad +f(u(t,i),v(t,i)), \\
\frac{\partial v(t,i)}{\partial t} &= -\frac{1-p^{\alpha}}{1-p^{-\alpha-1}}\left(\sum_{\substack{j\in G_m\\j\ne i}}\frac{p^{-m}v(t,j)}{|i-j|_p^{\alpha+1}} - v(t,i)\sum_{\substack{j\in G_m\\j\ne i}}\frac{p^{-m}}{|i-j|_p^{\alpha+1}}\right) \\
&\quad +g(u(t,i),v(t,i)), \\
\frac{\partial w(t,i)}{\partial t} &= h(u(t,i),v(t,i)).
\end{split}
\end{align}

See~\autoref{appe}. Using this discretization, we employ the ODE45 solver implemented in MATLAB to solve the system.
We begin by analyzing the purely temporal dynamics, assuming uniform concentrations across space ($x$-independent). This allows us to isolate and interpret the role of each parameter in the chemical reactions. Note that, in this case, there is no diffusion. In the following figures axes show dimensionless concentrations (see \autoref{remark_nondimensionalization}).

\begin{figure}[H]
    \centering
\includegraphics[width=0.45\linewidth]{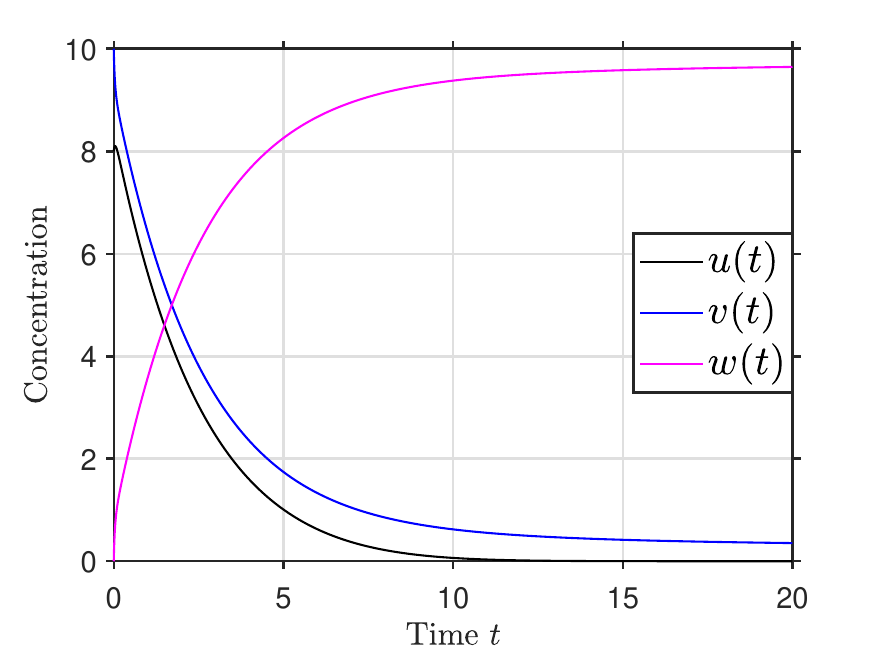}
\caption{Solution of the system ~\eqref{Sdiscret} without diffusion for $\beta = -0.2$, $\sigma = 1$, and $\eta = 1$.}
\label{sin_dif}
\end{figure}

In~\autoref{sin_dif}, we show the solution of the system without the diffusive part by fixing $ \beta =-0.2$, $\sigma = 1$, and $\eta = 1$ with the initial condition $(u(0),v(0),w(0))=(8,10,0)$. The graph shows the concentrations resulting from the chemical reaction. As time increases, the concentrations of \ce{CO3} and \ce{Ca} decrease while \ce{CaCO3} is produced. When \ce{Ca} is depleted, the amount of \ce{CaCO3} becomes constant. Since no additional \ce{Ca} or \ce{CO3} is supplied, \ce{CaCO3} increases until it stabilizes.

When the parameter $\sigma$, representing the initial amount of \ce{CO2}, is increased, 
the availability of \ce{CO2} in reaction~\eqref{Chemical1} rises, transiently enhancing the early 
production of \ce{HCO3}. This larger bicarbonate pool promotes reaction~\eqref{Chemical2}, in 
which \ce{Ca} combines with \ce{HCO3} to form \ce{CaCO3}. As a consequence, the precipitation of \ce{CaCO3} is accelerated.
In~\autoref{variation_sigm}, two cases are shown, with the remaining parameters and initial conditions being the same as in~\autoref{sin_dif}.
\begin{figure}[H]
\centering
\subfloat[$\sigma = 0.5$ \label{diox_menor}]{
  \includegraphics[width=0.45\linewidth]{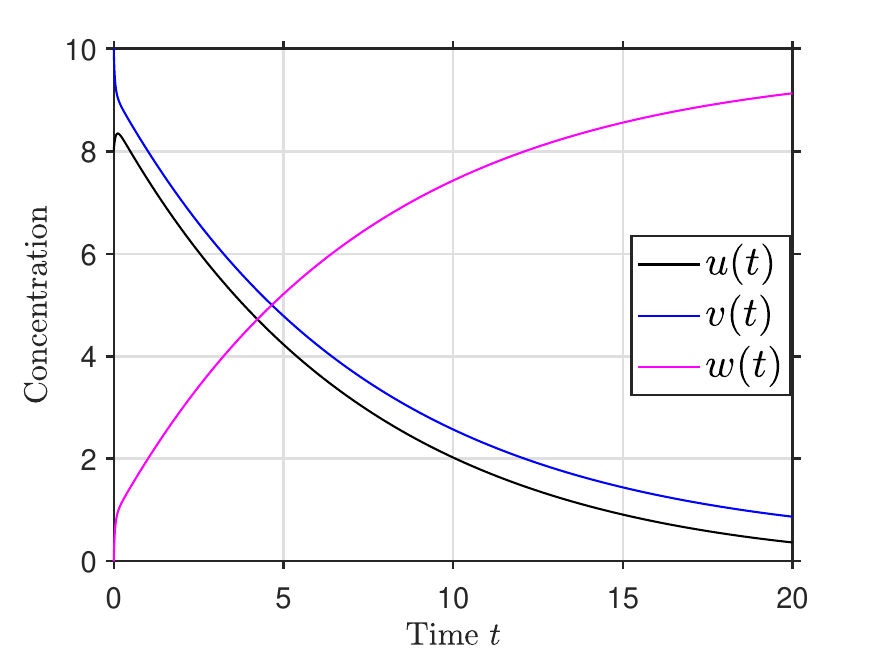}
}
\hfill
\subfloat[$\sigma = 2$ \label{diox_mayor}]{
  \includegraphics[width=0.45\linewidth]{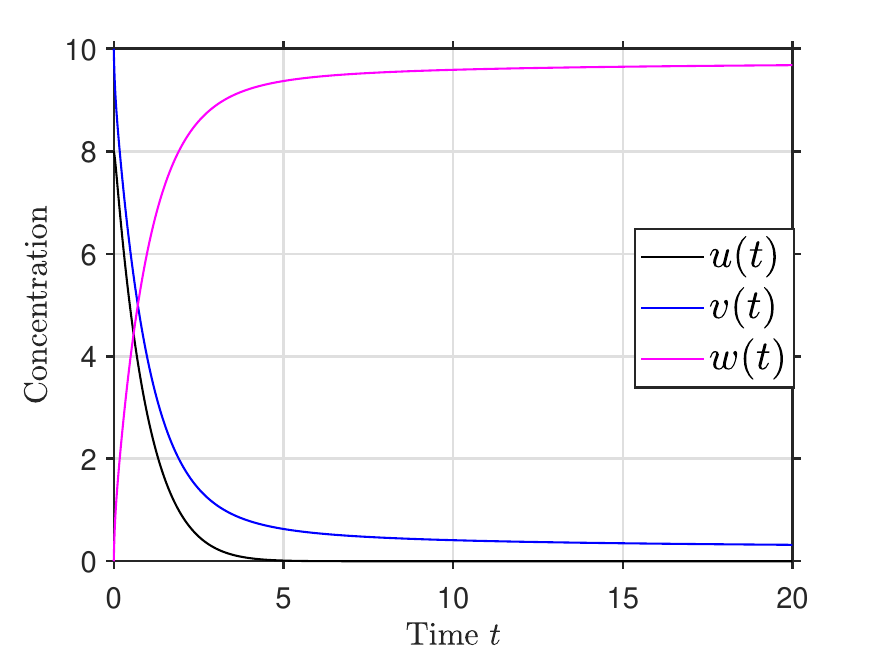}
}
\caption{The plots show the temporal evolution of the system for different values of $\sigma$, 
representing the initial \ce{CO2} concentration. A higher $\sigma$ increases the initial 
availability of \ce{CO2}, enhancing early \ce{HCO3} production and thus accelerating 
the precipitation of \ce{CaCO3}, while the total yield remains fixed by the initial calcium 
and carbonate stocks.}
\label{variation_sigm}
\end{figure}

For the hierarchical diffusion, we set $p=2$ for computational tractability. We take $\alpha=2$ as a common Laplacian-like benchmark: larger
$\alpha$ accelerates homogenization across the hierarchy, whereas smaller $\alpha$
preserves sharper inter-branch contrasts. With $p=2$ and $\alpha=2$ we use the
matrix~\eqref{alpha_2}. Biologically, $p=2$ corresponds to binary branching; more
generally, the branching factor equals the prime $p$.

\begin{remark}\label{rem:choice_of_p_m}
The prime number $p$ is a mathematical requirement for defining the $p$-adic field $\mathbb{Q}_p$ and its ultrametric topology. 
We use $p=2$ for simplicity. 
Larger primes can also be used—the same mathematical analysis applies—but they increase the number of branches and therefore the computational cost. 
Since real coral colonies have finitely many branches, this does not represent a limitation of the approach but only a practical consideration.
\end{remark}

As defined above, $d=d_2/d_1$. Measurements show that carbonate species have larger
diffusion coefficients than \ce{Ca^{2+}} \citep{POISSON1983,ZEEBE2011}, so $d<1$.
In the calcifying space, effective diffusivities are further reduced by tortuosity
and by intercellular junctions, which restrict paracellular transport
\citep{Tambutt2011}; these effects likely hinder \ce{Ca^{2+}} more than carbonate. Consequently, smaller $d$ reduces calcium mixing relative to
carbonate and accentuates branch-level heterogeneity; to foreground this regime we set
$d=0.1$ for the numerical solutions.

{\begin{remark}\label{parameters_simulations}
$\sigma<1$ or $\eta<1$ slows it. 
Unless stated otherwise, we fix $p=2$, \(\alpha=2\), $d=0.1$ $\sigma=\eta=1$ and $\beta=-0.2$.
\end{remark}

Supposing that the coral has two branches, we represent them by the first hierarchical level ($m=1$) of the $2$-adic tree. 
Each branch corresponds to a ball in $\mathbb{Z}_2$ of radius $\tfrac{1}{2}$—specifically, $B_{-1}(0)$ and $B_{-1}(1)$, centered at $0$ and $1$, respectively. 
We assume that the concentrations of \ce{CO3}, \ce{Ca}, and \ce{CaCO3} are spatially uniform within each branch—i.e., depending only on time on each branch—while possibly differing between branches. 
Accordingly, we solve system~\eqref{Sdiscret} for $m=1$, treating each ball $B_{-1}(i)$ as a well-mixed compartment coupled by diffusion through the hierarchical structure. 
\begin{figure}[H]
\centering
\subfloat[Concentration of \ce{CaCO3} given two branches\label{dos_bol}]{
  \includegraphics[width=0.45\linewidth]{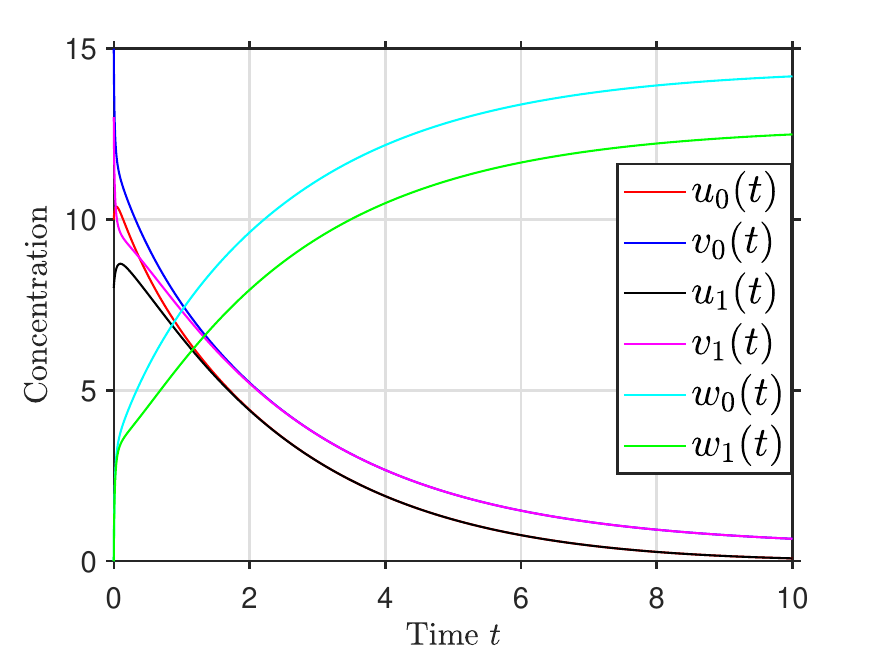}
}
\hfill
\subfloat[Concentration of \ce{CaCO3} given four branches\label{cuatro_bol}]{
  \includegraphics[width=0.45\linewidth]{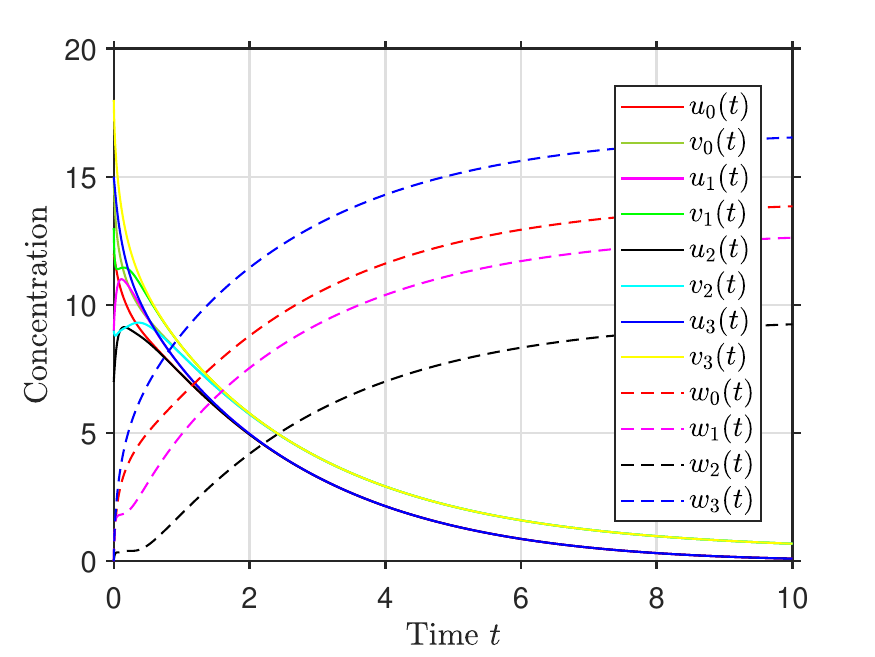}
}
\caption{The figures illustrate the concentration of \ce{CaCO3} assuming that the coral possesses two and four branches, respectively. Figure~\ref{dos_bol} shows the \ce{CaCO3} concentration when the coral has two branches, while~\autoref{cuatro_bol} presents the corresponding solution for a coral with four branches.}

\label{branches}
\end{figure}

To reveal inter-branch differences, the initial concentrations of calcium and carbonate are taken slightly different in each branch; otherwise, the symmetry of the system would produce identical trajectories.
The solution for this case is shown in~\autoref{dos_bol}, which illustrates the temporal evolution of the concentrations in the two balls, using the initial conditions 
\((u_0(0), u_1(0), v_0(0), v_1(0), w_0(0), w_1(0)) = (10, 8, 15, 13, 0, 0)\).
The values $u_i(t)$, $v_i(t)$, and $w_i(t)$ represent the concentrations of \ce{CO3}, \ce{Ca}, and \ce{CaCO3}, respectively, in the region corresponding to the support of the function $\boldsymbol{1}{B_{-1}(i)}$.

When the two branches further divide, resulting in four branches, the problem reduces to solving the system~\eqref{Sdiscret} taken $m = 2$. In this case, the functions $u$, $v$, and $w$ only depend on time within balls centered at $0, 1, 2,$ and $3$ (in base two) with radius $\frac{1}{4}$, and zero outside these balls.

The solution for this case is shown in~\autoref{cuatro_bol}, which illustrates the behavior of concentrations over time in each of the four regions. Each subscript $ i \in {0,1,2,3}$ corresponds to center of the respective ball of radius $1/4$, used to represent the branches of the coral. 

In all cases, we set the initial
\ce{CaCO3} to zero, i.e., $w(0)=0$ on every branch.

\subsection{Coral growth simulation}
Having established the dynamics of ion concentrations for a fixed number of branches, 
we now propose a coral growth simulation that incorporates diffusion processes. 
To achieve this, we first introduce a branching condition that determines when a new branch emerges, 
followed by a death condition that identifies when a branch no longer receives sufficient nutrients to sustain its growth. 
These two conditions are then combined to implement the simulation of a ramified coral structure. 
In what follows, we first describe the branching condition, then the death condition, 
and finally the numerical procedure used to simulate coral growth under these rules.

\subsubsection{Ramification condition}

Branching in the model arises from the internal chemical kinetics of the coral rather than from 
environmental drivers such as light, flow, or external feeding. 
The system is closed, representing only ion diffusion and reaction within the coenosarc network. 
As reactants are progressively consumed, the rate of calcium carbonate formation decreases, 
reflecting the depletion of dissolved calcium available for precipitation. 
Branching is defined by the condition $w(t) = v(t)$, which corresponds to the point where half of the initial 
calcium pool has been converted into solid \ce{CaCO3}.

\begin{figure}[H]
    \centering
    \includegraphics[width=0.45\linewidth]{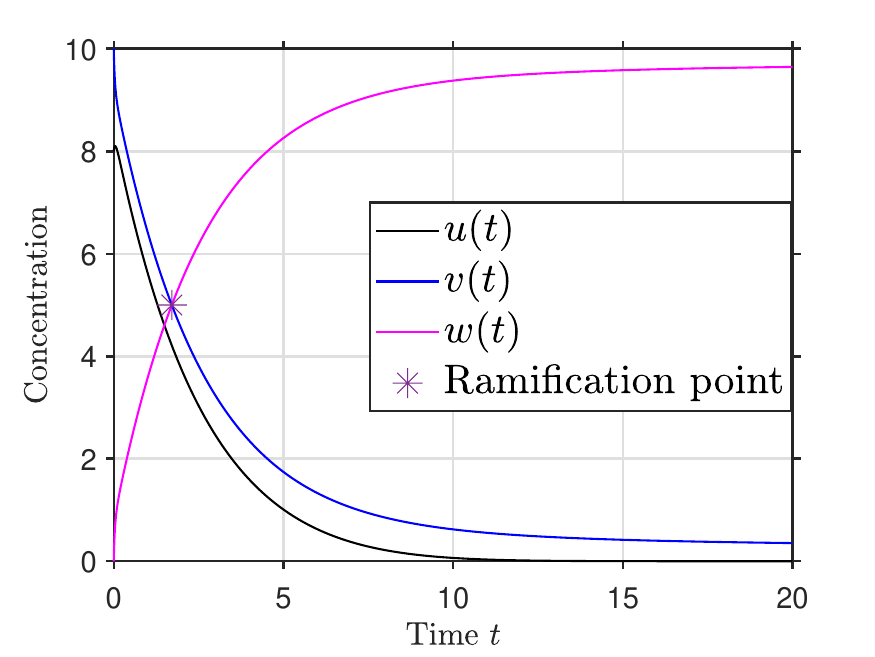}
    \caption{The star marks the time when the deposition rate first satisfies $w(t)=v(t)$.}
    \label{bifurcation}
\end{figure}

At this stage, calcium becomes the limiting reagent, and the reaction kinetics transition 
from a fast precipitation regime to a slower, diffusion-limited one. 
This condition thus provides a model-intrinsic criterion for internal resource limitation, 
representing the onset of spatial heterogeneity in calcification that drives the emergence 
of new branches within the coral structure. 
Numerically, this intersection is detected by locating the first time step at which $v(t)-w(t)$ changes sign, corresponding to the moment when $v(t)$ becomes smaller than $w(t)$.
In MATLAB, this is implemented by finding the first index $j$ such that $v_{\text{sol}}(j)-w_{\text{sol}}(j)<0$.

We note that this branching rule is phenomenological: it identifies a kinetic transition within a closed system rather than a direct biological threshold. 
Its sensitivity to initial ion concentrations or reaction rates could be further explored in future work to assess how local chemical conditions influence branch timing and morphology.

According to \autoref{variation_sigm}, increasing the value of the parameter \(\sigma\) causes the curves of \(v\) and \(w\) to intersect at an earlier time. 
This indicates that the branching event occurs sooner; hence, a higher initial amount of \ce{CO2} (reflected by a larger \(\sigma\)) leads to an earlier onset of branching.

\subsubsection{Halting condition}
In addition to the branching criterion, it is necessary to define a condition that determines when branch growth stops. 
This \textit{halting condition} reflects the physiological limitation of a branch that can no longer sustain calcification due to insufficient ion availability. 

The concept of calcium carbonate saturation has been extensively studied in coral calcification. In~\citep{gattuso_photosynthesis_1999}, the role of carbonate chemistry in regulating calcification and bicarbonate availability is analyzed. The effect of aragonite saturation on coral reef communities is explored in~\citep{silverman_effect_2007}, where elevated \ce{CO2} levels are shown to potentially induce coral dissolution. 

The saturation state is given by:
\begin{equation}
\Omega = \frac{[\ce{Ca^{2+}}][\ce{CO3^{2-}}]}{K_{sp}},
\end{equation}
where $K_{sp}$ is the solubility product constant of \ce{CaCO3}. Under optimal seawater conditions, \linebreak$K_{sp} \approx 6.65(\pm 0.12) \times 10^{-7} \, (\text{mol}/\text{Kg})^2$ see \citep{Mucci1980}). 

Following~\citep{yamamoto_threshold_2012}, we adopt $\Omega = 1$ as an operational threshold indicating 
that the system has reached chemical equilibrium within the closed calcifying medium. 
At this point, calcium availability is nearly exhausted and the precipitation rate becomes negligible, 
marking the halting point in our simulations. 
This interpretation differs from dissolution-based thresholds in open-seawater studies 
\citep{steiner_water_2018}, as our formulation represents a closed system where $\dot{w} \ge 0$ by construction.
The halting point corresponding to the curve in \autoref{sin_dif} is illustrated in \autoref{dead}.

\begin{figure}[H]
     \centering
     \includegraphics[width=0.5\linewidth]{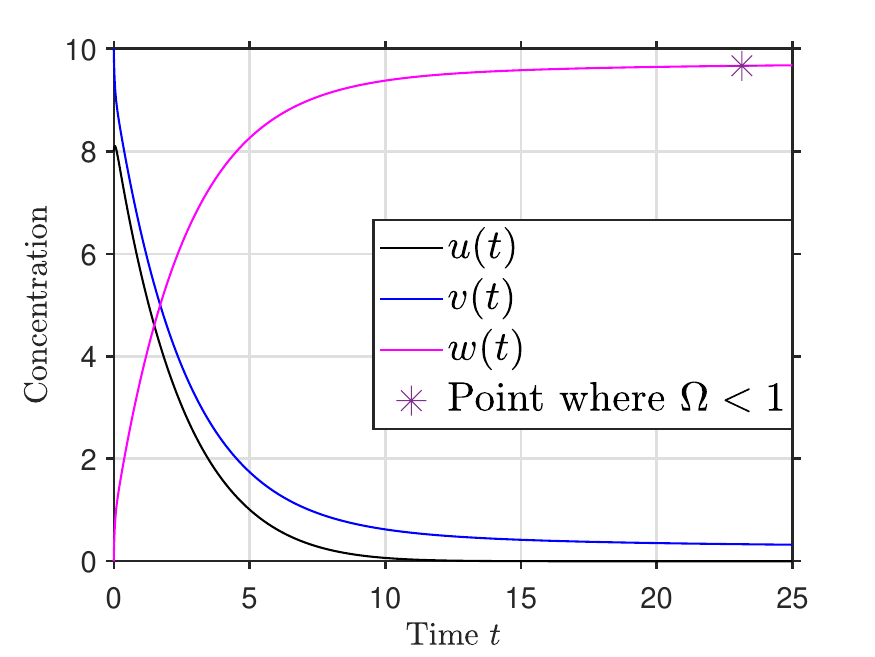}
     \caption{
The star marks the time at which the saturation index reaches $\Omega = 1$, 
indicating that calcium availability is nearly exhausted and precipitation effectively ceases. 
This threshold denotes kinetic exhaustion within the closed system rather than actual dissolution.
}    
s\label{dead}
\end{figure}

\subsubsection{Simulation of coral growth}

To simulate coral growth, we implemented the following event–driven procedure for $p = 2$:

\begin{enumerate}
    \item \textbf{Level \(m = 0\) (single branch).}
    The system is first solved under spatially uniform conditions on \(\mathbb{Z}_2\) (well–mixed), i.e., concentrations depend only on time. 
    The branching time \(t_b\) is defined as the first instant such that \(w(t_b) = v(t_b)\), at which the saturation index \(\Omega(t_b)\) is computed.

    \item \textbf{Branching test (level \(m = 0 \rightarrow 1\)).}
    If \(\Omega(t_b) \ge 1\), a bifurcation occurs, creating two daughter branches (\(m = 1\)) represented by balls of radius \(1/2\) centered at \(0\) and \(1\).

    \item \textbf{Initialization at birth (mass–conserving split).}
    At the branching time \(t_b\), daughter branches share the same chemical environment. 
    Their initial states are therefore drawn from the parent concentrations via a small random, mass–conserving split:
    \[
      x_0(0) = \theta\,x_{\mathrm{parent}}(t_b), \qquad 
      x_1(0) = (1-\theta)\,x_{\mathrm{parent}}(t_b), \quad \theta \in (0,1),
    \]
    applied to each \(x \in \{u, v\}\), while newborn branches start with \(w(0) = 0\).

    \item \textbf{Level \(m = 1\) (two branches).}
    The system is then solved independently on the two balls. 
    For each branch \(i \in \{0,1\}\), we detect the first time \(t_b^{(i)}\) for which \(w_i(t) = v_i(t)\) and evaluate \(\Omega_i(t_b^{(i)})\).
    A branch can bifurcate only if \(\Omega_i(t_b^{(i)}) \ge 1\).

    \item \textbf{Synchronization to the next level.}
    To maintain a hierarchy compatible with the Vladimirov operator on \(\mathbb{Z}_2\), 
    the simulation proceeds until all branches at level \(m\) that satisfy the bifurcation criterion have done so. 
    Branches meeting the condition produce two daughters; those that do not are retained or halted depending on their saturation level.

    \item \textbf{Initialization at the new level.}
    Each bifurcating branch \(i\) generates two daughters initialized by a mass–conserving random split of its state at \(t_b^{(i)}\) (as described in Step 3). 
    Set \(w=0\) for all newborn branches.

    \item \textbf{Iteration over levels.}
    Increase \(m \leftarrow m + 1\) and repeat Steps 4–6, 
    yielding \(2^m\) branches once all parent branches at level \(m-1\) have bifurcated. 
    The process stops when all branches fail the bifurcation criterion or meet a halting condition (\(\Omega_i < 1\)).
\end{enumerate}

\begin{remark}
\begin{itemize}
  \item[(i)] At bifurcation, daughter branches inherit nearly identical chemical states from the parent; small stochastic perturbations introduce asymmetry while conserving total mass. 
  Diffusion and reaction then drive differentiation among branches, while ion exchange between neighboring branches remains possible via the diffusion term.

  \item[(ii)] The synchronization step ensures that the number of branches remains a power of two (\(2^m\)), matching the discrete hierarchy required for the Vladimirov operator on the truncated \(p\)-adic tree. 
  This procedure is a numerical exploration of how \(p\)-adic representations capture transport between branches. 
  The advantage of the \(p\)-adic setting is its natural ability to quantify nonlocal diffusion through a hierarchical structure—something that a Euclidean domain cannot reproduce.

  \item[(iii)] In the graphical representation, branch length is proportional to the time required to reach the bifurcation condition. 
  Because branching times differ among branches, the resulting structure is not perfectly symmetric, although partial symmetry may arise depending on how the system is synchronized (using the first or last branch as reference).

  \item[(iv)] Each branch divides into two in the present case, but the method can be generalized to other primes \(p > 2\), allowing ternary or higher-order branching patterns (see~\autoref{rem:choice_of_p_m}). 
  Only the prime case is considered here; however, since any integer decomposes into prime powers, a general system could, in principle, be represented as a combination of \(p\)-adic components for different primes. 
  The coral-like figures were generated using the \texttt{L-py} framework.
\end{itemize}
\end{remark}

Figure~\ref{coral} shows the coral structure obtained from the simulation. 
The image was rendered using the L-Py library, which enables rule-based generation of branching structures. 
The length of each segment connecting two branching vertices represents the time elapsed before bifurcation—that is, the time between the branch’s origin and the point at which it divides.

\begin{figure}[H]
    \centering
\includegraphics[width=0.25\linewidth]{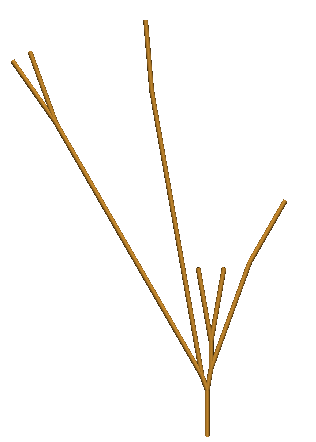}
    \caption{Simulated coral generated with $\sigma=1$, $\beta=-0.2$, $\eta=1$, $\alpha=2$ and $d=0.1$. Branch lengths correspond to total lifetime; the shortest-lived branch lasted $6.90514$ units, and the longest one $17.11237$}
    \label{coral}
\end{figure}
To analyze the influence of the parameters \(\sigma\) and \(\alpha\), two separate simulations were performed: 
one varying only \(\sigma\) and another varying only \(\alpha\), while keeping all other parameters 
as in~\autoref{coral}. 
To ensure that neither stochastic effects nor geometric differences interfered with the comparison, 
both the random value used for branch division and the inclination angles employed to draw the branches 
were the same as in~\autoref{coral}.

For smaller values of $\sigma$, the branching time increases, leading to longer branches before bifurcation. 
Although the simulated coral structures for $\sigma=0.5$(\autoref{coral_dioxid}) and $\sigma=1$(\autoref{coral}) appear similar in overall shape, 
the model predicts that lower $\sigma$ values produce more extended growth before the next branching event. 
In the rendered images, shorter branches are displayed with greater thickness by the L-Py library to enhance visibility, 
while longer ones appear thinner. 
Because all panels occupy the same fixed frame size and are not drawn to scale, 
the differences in branch length may appear visually reduced.

\begin{figure}[H]
    \centering
    \includegraphics[width=0.25\linewidth]{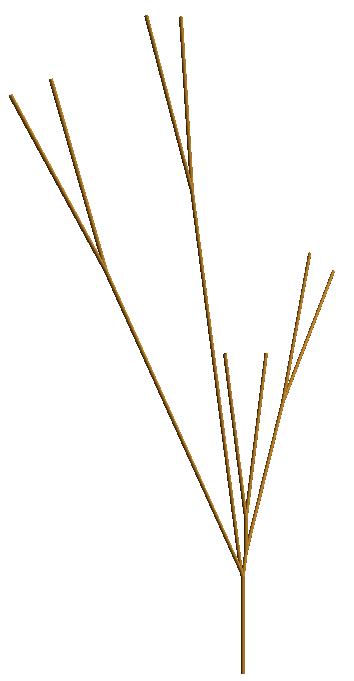}
    \caption{Simulated coral generated with $\sigma=0.5$, $\beta=-0.2$, $\eta=1$, $\alpha=2$ and $d=0.1$. Branch lengths correspond to total lifetime; the shortest-lived branch lasted $8.36579$ units, and the longest one $34.86969$}
    \label{coral_dioxid}
\end{figure}

The parameter $\alpha$ in the Vladimirov operator controls the rate at which diffusion decays with $p$-adic distance. 
For small $\alpha$, long-range interactions between distant branches are stronger, resulting in a more heterogeneous distribution of nutrients (superdiffusive regime). 
As $\alpha$ increases, the operator becomes more local, penalizing long-range transfers and promoting homogenization within each hierarchical level (subdiffusive regime). 
In the case $\alpha = 5$, the distribution is nearly uniform across branches, yielding almost identical branch lengths. 

From a biological standpoint, $\alpha$ can be interpreted as an effective index of ion-transport efficiency within the coral’s gastrovascular network. 
High $\alpha$ values correspond to efficient, well-connected transport among polyps, whereas low $\alpha$ values represent restricted or tortuous exchange that fosters local chemical heterogeneity.

When $\alpha$ is large, diffusion acts locally and efficiently equalizes the nutrient concentrations between neighboring branches. 
This effect is particularly evident at the first bifurcation, when the available nutrient pool is largest. 
At this stage, diffusion rapidly balances the concentrations between the two initial branches. 
Consequently, both branches reach the bifurcation threshold almost simultaneously, with nutrient concentrations that are nearly identical at that moment. 
Upon bifurcation, each parent branch divides its nutrient content randomly between its two daughters, but the sum of their concentrations equals that of the parent at the time of division. 
Since the two parent branches had almost identical concentrations, the total nutrient amounts assigned to each daughter pair are also very similar. 
When diffusion is then applied to the four branches, the Vladimirov operator acts locally within each sibling pair, rapidly balancing their concentrations. 
As a result, each pair of sister branches reaches an internal equilibrium close to equal nutrient distribution, as if the parent’s resources had been divided evenly. 
This mechanism repeats across successive generations, producing the overall symmetry observed in the simulations. 
Thus, the apparent homogeneity does not arise from long-range diffusion between distant branches, but from this recursive local equalization initiated at the first bifurcation under large~$\alpha$(see~\autoref{coral_alpha}).

Graphs~\autoref{coral},\autoref{coral_dioxid} and \autoref{coral_alpha} show coral structures generated from MATLAB data and rendered with L-py. 
Branch thickness is enhanced in shorter segments for visibility; the renderings are illustrative and not to scale.

It should be emphasized that the model does not aim to reproduce the precise three-dimensional geometry of coral colonies, such as growth orientation, branching angles, or overall symmetry. 
Rather, it generates a mathematical skeleton of branching events driven by the biochemical and diffusion dynamics. 

\begin{figure}[H]
    \centering
    \includegraphics[width=0.3\linewidth]{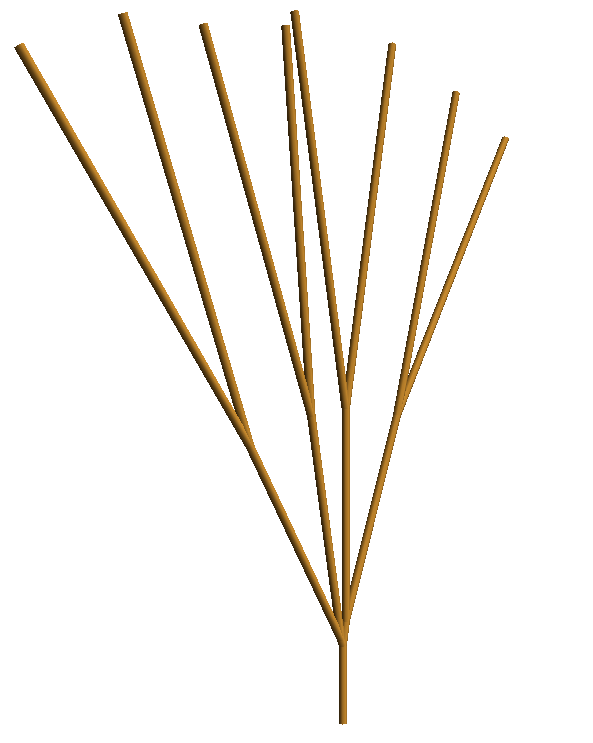}
    \caption{Simulated coral generated with $\sigma=1$, $\beta=-0.2$, $\eta=1$, $\alpha=5$ and $d=0.1$. Branch lengths correspond to total lifetime; the shortest-lived branch lasted $16.33612$ units, and the longest one $16.76954$}
    \label{coral_alpha}
\end{figure}

By bridging non-Archimedean analysis with nonlinear biological dynamics, this framework enriches the mathematical modeling of complex growth processes and expands the theoretical toolkit for analyzing hierarchical morphogenesis in natural systems.
Future developments could incorporate photosynthetic activity, couple the system with oceanic flow fields, and calibrate parameters using empirical coral growth data—thereby further integrating mathematical the ecological application.

%------------Conclusions--------------------%

\section*{Concluding Remarks}
We constructed a nonlinear reaction–diffusion system over $\mathbb{Z}_p$ to model internal calcification processes in coral colonies. 
The Vladimirov operator describes nonlocal diffusion within the coenosarc network, capturing the hierarchical and self-similar structure of branching corals. 
The model incorporates chemically grounded reactions between calcium and bicarbonate ions that drive the precipitation of calcium carbonate (\ce{CaCO3}), providing a quantitative framework to analyze internal chemical dynamics.

A current limitation of the framework is its reliance on a single prime $p$, which enforces a perfectly regular binary branching structure. 
Although the case $p=2$ simplifies computation and captures the essential hierarchical organization, it cannot reproduce the irregular branching patterns observed in natural coral colonies. 
Extending the model to mixed or irregular ultrametric trees would better approximate biological variability and represents a promising direction for future work.

A branching rule based on local \ce{CaCO3} accumulation, together with a halting condition linked to saturation thresholds, reproduces morphologies consistent with coral-like growth. 
Although these thresholds are externally defined, they are chemically motivated, reflecting resource depletion and the transition from reaction-dominated to diffusion-limited kinetics. 
Numerical simulations demonstrate how internal reaction–diffusion interactions within the coenosarc regulate nutrient redistribution, calcification rates, and the hierarchical architecture of the resulting structures.

While the model treats water activity as constant and neglects reversible or pH-dependent reactions, it provides a coherent first approximation of the internal kinetics governing coral calcification in a closed system. 
This framework establishes a foundation for future extensions incorporating photosynthetic activity, external ion exchange, or coupling with hydrodynamic transport.

Beyond coral systems, the proposed $p$-adic framework offers a new mathematical tool for studying morphogenesis in hierarchical or ultrametric domains. 
Possible extensions include parameter estimation from experimental datasets and reef-scale simulations relevant to conservation and climate resilience. 
By integrating non-Archimedean diffusion theory with nonlinear biochemical dynamics, the model advances the mathematical representation of internal calcification and hierarchical morphogenesis in natural systems.

\appendix
\section{Basic Facts on \texorpdfstring{$p$}{p}-Adic Analysis}\label{appe}

This section presents some results in $p$-adic analysis. First, we introduce the field of $p$-adic numbers and then describe some function spaces used to construct our biological model in~\autoref{Biochemical_model}. For readers interested in exploring these topics further, classical references include \citep{vladimirov_p-adic_1994}, \citep{taibleson_fourier_1975}, \citep{albeverio_theory_2010}, among others.
\subsection*{\texorpdfstring{The field of $p$-adic numbers}{The field of p-adic numbers}}

We fix a prime number $p$. In the field of rational numbers $\mathbb{Q}$, we define the \textit{$p$-adic absolute value} $|\cdot|_p:\mathbb{Q} \to \mathbb{R}$ by  

\begin{equation}\label{abs_value}
|x|_{p}=
\begin{cases}
0, & \text{if }x=0,\\
p^{-l}, & \text{if } x=p^{l}\dfrac{a}{b},
\end{cases}
\end{equation}

where $a$ and $b$ are integers coprime with $p$, and $a \cdot b \neq 0$, here $l$ is called the order of $x$ and is denoted as $\ord(x)$, in the case when $x=0$ we define $\ord(0):=+\infty$.  

This function satisfies the following properties for all $x, y \in \mathbb{Q}$:  
\begin{itemize}
    \item \textbf{Non-negativity:} $|x|_p \geq 0$, and $|x|_p = 0$ if and only if $x = 0$.
    \item \textbf{Multiplicativity:} $|xy|_p = |x|_p |y|_p$.
    \item \textbf{Ultrametric Inequality:} $\no{x+y}_p\le\max\set{\no{x}_p,\no{y}_p}$.
\end{itemize}
The ultrametric inequality implies the usual triangle inequality, thus confirming that $\no{\cdot}_p$ defines an absolute value on $\Q$.

The completion of $\mathbb{Q}$ relative to this absolute value \eqref{abs_value} is called the field of $p$-adic numbers and is denoted by $\mathbb{Q}_p$.  
Any $p$-adic number $x\neq 0$ can be written uniquely in the form:
\begin{equation}\label{p_serRep}
x=p^{ord(x)}\sum_{j=0}^{\infty}x_{j}p^{j},
\end{equation}

where $x_{j}\in\{0,1,2,\dots,p-1\}$ and $x_{0}\neq0$. By using this expansion,
we define \textit{the fractional part }$\{x\}_{p}$\textit{ of }$x\in
\mathbb{Q}_p$ as the rational number
\begin{equation*}
\{x\}_{p}=%
\begin{cases}
0 & \text{if }x=0\text{ or }\ord(x)\geq0\\
& \\
p^{\ord(x)}\sum_{j=0}^{-\ord(x)-1}x_{j}p^{j} & \text{if }\ord(x)<0.
\end{cases}
\end{equation*}
Also note that any $x\in\mathbb{Q}_p\smallsetminus\left\{  0\right\}$
can be represented uniquely as $x=p^{ord(x)}v$ where
$\no{v}_{p}=1$.

\subsection*{\texorpdfstring{Topology of $\mathbb{Q}_p$}{Topology of Qp}}

For $r\in\mathbb{Z}$, denote by $B_{r}(a)=\{x\in\mathbb{Q}_p;|x-a|_{p}\le
 p^{r}\}$ \textit{the ball of radius }$p^{r}$ \textit{with
center at} $a\in\mathbb{Q}_p$. 

Due to the ultrametric property, the geometry of $\mathbb{Q}_p$ differs drastically from that of the real line. For example, the unit ball centered at $0$, denoted by $\mathbb{Z}_p$, is both open and closed, and it forms a ring under the operations inherited from the field. This ring is called the \textit{ring of $p$-adic integers}. 

As a consequence of the ultrametricity, the space has the following properties:
\begin{itemize}
    \item Two balls in $\mathbb{Q}_p$ are either disjointed or contained in the other. 
    \item A subset of $\mathbb{Q}_p$ is
compact if and only if it is closed and bounded in $\mathbb{Q}_p$.
    \item In $\mathbb{Q}_p$
balls are compact subsets and $\left(  \mathbb{Q}_p%
,|\cdot|_{p}\right)$ becomes a locally compact topological space.
    \item $\left(  \mathbb{Q}_p,| \cdot |_{p}\right)$ is
totally disconnected, i.e. the only connected \ subsets of $\mathbb{Q}_p$ are the empty set and the points.
\end{itemize}

\subsection*{Some function spaces}
This section introduces some function spaces defined on $\q$.

\subsubsection*{The Bruhat-Schwartz space}

A complex-valued function $\varphi:\q\rightarrow \C$ is \textit{locally constant} if for any $x\in\mathbb{Q}_p$ there exists an integer $r=r(x)\in\mathbb{Z}$ such that
\begin{equation}
\varphi(x+x^{\prime})=\varphi(x)\text{ for any }x^{\prime}\in B_{r}.
\label{local_constancy}
\end{equation}

$\varphi$ is a \textit{Bruhat-Schwartz (test)} function if it is locally constant with compact support. These functions are linear combinations of characteristic functions of balls. The set containing them forms a $\mathbb{C}$-vector space, which we denote by $\tes(\mathbb{Q}_p):=\tes$.

Because $\varphi\in\tes(\mathbb{Q}_p)$ is locally constant, for every $x$ there exists $r(x)\in\Z$ satisfying ~\eqref{local_constancy}. Since it has compact support, it is possible to choose the largest $r(x)$, such a number denoted by $r=r(\varphi)$ is called \textit{the local constancy exponent (or constancy parameter) of} $\varphi$.

If $\varphi:\mathbb{Q}_p\rightarrow\R$, the space $\tes$ is a $\R$-vector space. We will use a subscript to denote this space, writing it as $\tes_{\R}$.
By $C(\q, \mathbb{R})$, we denote the vector space of all the continuous real-valued functions defined on $\q$.

This space is analogous to the Schwartz space in real analysis and is a natural domain for $p$-adic Fourier analysis.

\subsubsection*{\texorpdfstring{$L^{\rho}$ spaces}{Lrho spaces}}

$(\mathbb{Q}_p,+)$ is a locally compact topological group. By the Haar theorem, there exists a Haar measure in $\mathbb{Q}_p$, which is denoted by $dx$. This measure is invariant under translations, i.e.,
$d(x+a)=dx$. With this measure, we have the notion of integral in $\mathbb{Q}_p$. If we normalize this measure by the condition
$\int_{\mathbb{Z}_{p}}dx=1$, then $dx$ is unique.

For $\rho\in\lbrack 1,\infty)$, $L^{\rho}=L^{\rho}\left(
\mathbb{Q}_p\right),$ is the vector space of all the complex-valued functions $g:\q\rightarrow\C$ satisfying
\begin{equation*}
{\displaystyle\int\limits_{\mathbb{Q}_p}}
\left\vert g\left(  x\right)  \right\vert ^{\rho}dx<\infty.
\end{equation*}
We endow this space with the norm
\begin{equation*}
    \Nor{g}_{\rho}=\left({\displaystyle\int\limits_{\mathbb{Q}_p}}
\left\vert g\left(  x\right)  \right\vert ^{\rho}dx\right)^{1/{\rho}},
\end{equation*}
with this norm, the $L^{\rho}$ space becomes a Banach space \citep{taibleson_fourier_1975}(Chapter 2).

\subsection*{The Fourier transform}

Denote by $S$ the unit complex circle. The function $\chi_{p}:\p{\mathbb{Q}_p,+}\rightarrow \p{S,\cdot}$ defined as $\chi_p(y)=\exp(2\pi i\{y\}_{p})$ is an additive character on $\mathbb{Q}_p$, i.e., it is a continuous homomorphism of groups. 

The Fourier transform $\mathcal{F}:\tes\rightarrow\tes$ is defined as $\mathcal{F}(\varphi)=\widehat{\varphi}$ where,
\begin{equation*}
\widehat{\varphi}(\kappa)=
{\displaystyle\int\limits_{\mathbb{Q}_p}}
\chi_{p}(\kappa\cdot x)\varphi(x)dx\quad\text{for }\kappa\in\mathbb{Q}_p \text{ and } \varphi\in\tes.
\end{equation*}

This is a $\C$-linear isomorphism from $\tes(\mathbb{Q}_p)$ onto itself satisfying
\begin{equation}
(\mathcal{F}(\mathcal{F}(\varphi)))(x)=\varphi(-x), \label{Eq_FFT}%
\end{equation}
see e.g. \citep{albeverio_theory_2010}(Section 4.8). It extends to a unitary operator on $L^2(\mathbb{Q}_p)$, preserving inner products and mapping $L^2$ onto itself \citep{taibleson_fourier_1975}(Theorem 2.3). The extension and its inverse are given by
\begin{align*}  
\mathcal{F}(\varphi)(\kappa)&=\lim\limits_{r\rightarrow\infty}\int_{B_r}\varphi(x)\chi_p(\kappa\cdot x)dx\\
\mathcal{F}^{-1}(\varphi)(\kappa)&=\lim\limits_{r\rightarrow\infty}\int_{B_r}\varphi(x)\chi_p(-\kappa\cdot x)dx.
\end{align*}
\citep{vladimirov_p-adic_1994}(Chapter 1, section VII.4). We will also use the notation $\mathcal{F}_{x\rightarrow\kappa}(\varphi)$ for the Fourier transform of $\varphi(x)$. 

\subsection*{The Vladimirov operator}\label{Vladop}

Because the ordinary derivative is not defined in a meaningful way for locally constant functions, it is necessary to define an analog of the differentiation operator. To this end, we work with pseudo-differential operators. One such operator is the Vladimirov operator.

Given $\alpha>0$, for $u\in\tes(\mathbb{Q}_p)$, the Vladimirov operator denoted by ${\boldsymbol{D}}^{\alpha}$ is defined as
\begin{equation}
   \boldsymbol{D}^\alpha u(x) := \frac{1-p^{\alpha}}{1-p^{-\alpha-1}} \int_{\mathbb{Q}_p} \frac{u(y) - u(x)}{|x-y|_p^{1+\alpha}} \, dy,
\end{equation}

${\boldsymbol{D}}^{\alpha}$ computes a weighted average of the difference between $u$ at $x$ and all other points, and thus it is considered analogous to the Laplacian. Additionally, it is possible to rewrite the Vladimirov operator as

\begin{equation}\label{Vla1def}
    {\boldsymbol{D}}^{\alpha}u(x) =\mathcal{F}^{-1}_{k\rightarrow x}\left(|k|^{\alpha}_p\mathcal{F}_{x\rightarrow k}u \right).
\end{equation}

\citep[Chapter 2, section IX]{vladimirov_p-adic_1994}. In the case of $\alpha=2$, we have 
\begin{equation}\label{Vla2def}
    \mathcal{F}_{x\rightarrow k}({\boldsymbol{D}}^{2}u(x)) =|k|^{2}_p\mathcal{F}_{x\rightarrow k}(u(x)).
\end{equation}
Except for the sign, the last equation coincides with the equation that satisfies the second derivative operator when $u$ belongs to the classical real Schwartz space. Also, analog properties satisfied by the classical Laplacian hold for this operator \citep{kochubei_vladimirovtaibleson_2023}.

We are interested in ${\boldsymbol{D}}^{\alpha}$ acting on balls, specifically on the unit ball $\mathbb{Z}_p$. The restriction of this operator to balls was previously studied in \citep[Chapter 2, Section X.5]{vladimirov_p-adic_1994} and \citep{bikulov_complete_2019}. The case of $\mathbb{Z}_p$, was explored in \citep{chacon-cortes_turing_2023}. Let $u(x) \in \tes_{\R}(\mathbb{Z}_p)$ and $x \in \mathbb{Z}_p $. Then,
\begin{align*}
{\boldsymbol{D}}^{\alpha} u(x) 
=\frac{1-p^{\alpha}}{1-p^{-\alpha-1}}\int\limits_{\mathbb{Z}_p}\frac{u(y)-u(x)}{\no{x-y}_p^{\alpha+1}}dy+\frac{p^{\alpha}(p-1)}{p^{\alpha+1}-1}u(x)\\
\end{align*}

Thus, the restriction of ${\boldsymbol{D}}^{\alpha}$ to $\tes_{\R}(\z)$ satisfies 
\begin{equation*}
   \left({\boldsymbol{D}}^{\alpha}-\mu\right)u(x)=\frac{1-p^{\alpha}}{1-p^{-\alpha-1}}\int\limits_{\mathbb{Z}_p}\frac{u(y)-u(x)}{\no{x-y}_p^{\alpha+1}}dy,
\end{equation*}
where $\mu=\frac{p^{\alpha}(p-1)}{p^{\alpha+1}-1}$. From now on we will work with the operator $\overline{\boldsymbol{D}}^{\alpha}:= \left({\boldsymbol{D}}^{\alpha}-\mu\right)$.

In 2002, Kozyrev started the p-adic wavelet theory with his paper \citep{kozyrev_wavelet_2002}, where he introduced a basis of complex-valued wavelets with compact support in $L^2(\mathbb{Q}_p)$. The elements of the basis have the form 
\begin{equation}
\Psi_{r,j,n}(x)=p^{-\frac{r}{2}}\chi_p\left(p^{r-1}jx\right)\mathbf{1}_{B_{r}\left(p^{-{r}}n\right)}(x)
\end{equation}
where $r\in\mathbb{Z}$, $j\in\set{1,\dots,p-1},$ and $n$ runs through a fixed set of representatives of $\mathbb{Q}_p/\mathbb{Z}_p$. These functions are eigenfunctions of the operator ${\boldsymbol{D}}^{\alpha}$ with eigenvalue $p^{\left(1-r\right)\alpha}$ \citep{khrennikov_ultrametric_2018}(Theorem 3.29).

In \citep{bikulov_complete_2019}, \citep{chacon-cortes_turing_2023}, and \citep{zuniga-galindo_non-archimedean_2021}, it is considered the eigenvalue problem in balls. In the unit ball, the eigenvalue problem is given by
\begin{equation}\label{eigenv2}
\begin{cases}
    \overline{\boldsymbol{D}}^{\alpha}u(x)=ku(x) & x\in \z, \quad t\ge 0\\
u\in L^2\left(\mathbb{Z}_p\right).
\end{cases}
\end{equation}

In \citep{zuniga-galindo_eigens_2022}, the author showed that the set of functions
\begin{equation}
    \set{\boldsymbol{1}_{\mathbb{Z}_p}}\bigcup N_{r,j,n}
\end{equation}

where $N_{r,j,n}=\bigcup\limits_{r,j,n}\left\{\Psi_{r,j,n}(x); j\in\set{1,\dots,p-1}, r\le 0 , n\in p^r\mathbb{Z}_p\cap \mathbb{Q}_p/\mathbb{Z}_p\right\}$
is an orthonormal basis of $L^2(\mathbb{Z}_p)$, and functions belonging to $N_{r,j,n}$ are eigenfunctions of $\overline{\boldsymbol{D}}^{\alpha}$ with eigenvalue $p^{(1-r)\alpha}$.

\subsubsection*{Discretized system}
We take the set $G_m$ of $p$-adic integers of the form:
\begin{align*}
i = i_0 + i_1 p + \dots + i_{m-1} p^{m-1},
\end{align*}
where $i_k \in \{0,1,\dots,p-1\}$ for each $k \in \{0,1,\dots,m-1\}$.

Let $\boldsymbol{1}_{B_{-m}(i)}(x)$ be the characteristic function of the ball centered at $i \in G_m$ with radius $p^{-m}$. The elements of this ball have the form $i + p^my$ with $y\in\mathbb{Z}_p$. Note that for $i \neq j$, the balls ${B_{-m}}(i)$ and ${B_{-m}}(j)$ are disjoint. In fact $\no{i-j}_p>p^{-m}$.

By $X_m$, we denote the $\mathbb{R}$-vector space consisting of all the test functions in the form:
\begin{equation*}
\varphi(x) = \sum_{i\in G_m}\varphi(i)\boldsymbol{1}_{B_{-m}}{\left(i\right)}.
\end{equation*}
where $\varphi(i) \in \mathbb{R}$. This is a Banach space with the supremum norm.

Now, consider the following system:
\begin{align}\label{System_app}
\begin{split}
    \frac{\partial{u}({t},x)}{\partial {t}}&=-\boldsymbol{\overline{D}}^{\alpha}{u}({t},x)-{u}({u}-{v}+\sigma-\beta)\\
    \frac{\partial {v}({t},x)}{\partial {t}}&=-d\boldsymbol{\overline{D}}^{\alpha}{v}({t},x)-\eta{v}({v}-{u}+\beta)^2\\
    \frac{\partial {w}({t},x)}{\partial {t}}&=\eta{v}({v}-{u}+\beta)^2,
\end{split}
\end{align}

The Vladimirov operator, when restricted to functions in the space $X_m$, is represented by the matrix $L^{\alpha} = [L_{ij}^{\alpha}]$ of size $\#G_m \times \#G_m$, where:

 \begin{equation}\label{eq:Lij:A}
    L_{ij}^{\alpha}=
    \begin{cases}
        \frac{(1-p^{\alpha})p^{-m}}{1-p^{-\alpha-1}}\frac{1}{\no{i-j}_p^{\alpha+1}} &\text{ if } i\neq j\\
       -\frac{(1-p^{\alpha})p^{-m}}{1-p^{-\alpha-1}}\displaystyle\sum_{\substack{k\in G_m\\k\ne i}}\frac{1}{|i-k|_p^{\alpha+1}} &\text{ if } i= j.
    \end{cases}
\end{equation}
In the particular case when $\alpha=2$
\begin{equation}\label{eq:Lij:B}
    L_{ij}^{2}=
    \begin{cases}
        \frac{p^{3-m}(1-p^2)}{p^3-1}\frac{1}{\no{i-j}_p^{3}} &\text{ if } i\neq j\\
       -\frac{p^{3-m}(1-p^2)}{p^3-1}\displaystyle\sum_{\substack{k\in G_m\\k\ne i}}\frac{1}{|i-k|_p^{3}} &\text{ if } i= j.
    \end{cases}
\end{equation}

Equation~\eqref{eq:Lij:B} provides the discrete matrix representation of the restricted Vladimirov operator 
$D^{\alpha}$ acting on the finite-dimensional space $X_m$ of locally constant functions supported on disjoint 
$p$-adic balls of radius $p^{-m}$. Each index $i \in G_m$ labels one of these balls, which in our biological 
interpretation correspond to the coral branches.

The off-diagonal entries $L^{\alpha}_{ij}$ ($i \neq j$) quantify the nonlocal diffusive coupling between 
branches $i$ and $j$, with an intensity that decays as a power law of the $p$-adic distance 
$|i - j|_p^{-(\alpha + 1)}$. The prefactor 
\[
\frac{(1 - p^{\alpha}) p^{-m}}{1 - p^{-\alpha - 1}}
\]
arises from the normalization of the Vladimirov integral with respect to the Haar measure on $\mathbb{Z}_p$, ensuring 
that the discretization preserves the correct scaling and nonlocal structure of the continuous operator.

The diagonal terms are defined so that each row of $L^{\alpha}$ sums to zero, i.e.,
\[
\sum_{j \in G_m} L^{\alpha}_{ij} = 0,
\]
which guarantees conservation of total mass under diffusion. This condition is directly analogous to the 
construction of the discrete Laplacian in Euclidean lattices, where diagonal elements compensate for all 
outgoing fluxes to neighboring nodes.

Consequently, $L^{\alpha}$ acts as the generator of diffusion on the truncated $p$-adic tree of depth $m$, 
providing a consistent finite-dimensional approximation of the Vladimirov operator on $\mathbb{Z}_p$. In the context of our model, this operator governs the nonlocal exchange of ions between coral branches in the hierarchical structure.

Consider the functions $u,v$ belonging to $X_m$, then they take the form:
  \begin{align*}
   u(t,x) &= \sum\limits_{i\in G_m}u(t,i)\boldsymbol{1}_{B_{-m}(i)}(x), \\
   v(t,x) &= \sum\limits_{i\in G_m}v(t,i)\boldsymbol{1}_{B_{-m}(i)}(x),  \\
   \end{align*}
Since, for $i\neq j$ belonging to $G_m$, $\boldsymbol{1}_{B_{-m}(i)}(x) \cdot \boldsymbol{1}_{B_{-m}(j)}(x) = 0 $, a straightforward calculation gives:
 \begin{align*}
&  f(u(t,x),v(t,x)) = f\left(\sum\limits_{i\in G_m}u(t,i)\boldsymbol{1}_{B_{-m}(i)}(x), \sum\limits_{i\in G_m}v(t,i)\boldsymbol{1}_{B_{-m}(i)}(x)\right) \\
  &= -\sum\limits_{i\in G_m}{u(t,i)}({u(t,i)}-{v(t,i)}+\sigma-\beta) )\boldsymbol{1}_{B_{-m}(i)}(x)\\
   &= \sum\limits_{i\in G_m}f(u(t,i),v(t,i))\boldsymbol{1}_{B_{-m}(i)}(x). 
   \end{align*}

For $g$ and $h$, we can do similar calculations. Hence, functions $f(u,v), g(u,v)$ and $h(u,v)$ are represented as vector-valued functions: 
$F = [f(u(t,i), v(t,i))]_{i \in G_m} $, $G = [g(u(t,i), v(t,i))]_{i \in G_m} $ and $H = [h(u(t,i), v(t,i))]_{i \in G_m}$, with
\begin{align*}
       f(u(t,i),v(t,i)) &= -u(t,i)   (u(t,i) - v(t,i) + \sigma- \beta )\\
       g(u(t,i),v(t,i)) &= -\eta  v(t,i)  ( v(t,i) -u(t,i) + \beta)^2\\
       h(u(t,i),v(t,i)) &= \eta v(t,i)  ( v(t,i) -u(t,i) + \beta)^2;
   \end{align*}
The discretized version of~\eqref{System_app} consists of solving a system of $3p^m$ ordinary differential equations of the form:

\begin{equation}\label{discr}
\begin{aligned}
\frac{\partial u(t,i)}{\partial t} &= 
- \frac{1 - p^\alpha}{1 - p^{-\alpha - 1}} \left( 
\sum_{\substack{j \in G_m \\ j \ne i}} \frac{p^{-m} u(t,j)}{|i - j|_p^{\alpha + 1}} 
- u(t,i) \sum_{\substack{j \in G_m \\ j \ne i}} \frac{p^{-m}}{|i - j|_p^{\alpha + 1}} 
\right)
+ f(u(t,i), v(t,i)), \\[6pt]
\frac{\partial v(t,i)}{\partial t} &= 
- \frac{1 - p^\alpha}{1 - p^{-\alpha - 1}} \left( 
\sum_{\substack{j \in G_m \\ j \ne i}} \frac{p^{-m} v(t,j)}{|i - j|_p^{\alpha + 1}} 
- v(t,i) \sum_{\substack{j \in G_m \\ j \ne i}} \frac{p^{-m}}{|i - j|_p^{\alpha + 1}} 
\right)
+ g(u(t,i), v(t,i)), \\[6pt]
\frac{\partial w(t,i)}{\partial t} &= 
h(u(t,i), v(t,i)).
\end{aligned}
\end{equation}

In matrix form, this is equivalent to solving 
\begin{equation}\label{Matrix}
\begin{cases}
\displaystyle \frac{d\mathbf{u}}{dt} = L^{\alpha} \mathbf{u}(t) + F, \\[10pt]
\displaystyle \frac{d\mathbf{v}}{dt} = L^{\alpha} \mathbf{v}(t) + G, \\[10pt]
\displaystyle \frac{d\mathbf{w}}{dt} = H,
\end{cases}
\end{equation}
where $\mathbf{u}(t)=[u(i,t)]_{i\in G_m}, \quad\mathbf{v}(t)=[v(i,t)]_{i\in G_m}$ and $ \mathbf{w}(t)=[w(i,t)]_{i\in G_m}$. 

\begin{proposition}
The coupled system of ordinary differential equations~\eqref{Matrix} with initial conditions \( \mathbf{u}(t_0) = \mathbf{u}_0 \), \( \mathbf{v}(t_0) = \mathbf{v}_0 \), \( \mathbf{w}(t_0) = \mathbf{w}_0 \), admits a unique solution
\[
(\mathbf{u}, \mathbf{v}, \mathbf{w}) \in C^1((t_0 - \delta, t_0 + \delta), \mathbb{R}^{3n}).
\] with $\delta>0$.
\end{proposition}

\begin{proof}
The matrix \( L^\alpha \) defines a continuous linear operator. The functions \( f, g, h \) are smooth polynomials in \( u \) and \( v \), thus \( F, G, H \) are locally Lipschitz functions from \( \mathbb{R}^{2n} \) to \( \mathbb{R}^n \). Therefore, the right-hand side of the system defines a local Lipschitz vector field on \( \mathbb{R}^{3n} \), and the result follows the Picard–Lindelöf Theorem.
\end{proof}

This discretized system corresponds to the numerical implementation used in~\autoref{Numerical_sol}, where coral branching is simulated dynamically.

%------------Acknowledments--------------------%

\section*{Funding and conflict of interest}
The research leading to these results received funding from Secihti-Mexico under the Grant Agreement Fronteras 2019-217367.\\
The authors have no conflicts of interest to declare relevant to this article's content.

%------------References------------------------%
\clearpage
\bibliographystyle{elsarticle-num}
\bibliography{references1}
\end{document}